\documentclass[a4paper]{amsart}

\usepackage{graphicx, psfrag, layout, appendix, subfigure}
\usepackage{fancyhdr}

\usepackage{amssymb, amsthm, amsmath, amscd, amsfonts, mathabx, mathrsfs, latexsym}

\usepackage{enumitem}
\usepackage{color}
\usepackage{url}

\usepackage{setspace}
\usepackage{changepage}

\usepackage{leftidx}

\usepackage{bbold}

\usepackage[active]{srcltx}

\usepackage{calc}

\usepackage{pb-diagram,pb-xy}

\usepackage{times}

\usepackage{verbatim}
\usepackage[all,cmtip,arrow,matrix,curve,tips,frame]{xy}
\usepackage[
	setpagesize=false,
	colorlinks=true,
	linkcolor=blue,
	pdfencoding=auto,
	psdextra,
]{hyperref}
\usepackage{cleveref}
\xyoption{matrix}

\usepackage{xcolor}
\hypersetup{
    colorlinks,
    linkcolor={red!50!black},
    citecolor={blue!50!black},
    urlcolor={blue!80!black}
}

\usepackage{ifthen}
\usepackage{color}
\usepackage{amsxtra}
\usepackage{amstext}
\usepackage{stmaryrd}
\usepackage{mathrsfs}
\usepackage{epsfig}
\usepackage{pifont}
\usepackage{charter}
\usepackage{avant}
\usepackage{courier}
\usepackage[T1]{fontenc}
\usepackage{textcomp}
\usepackage{bbm}
\usepackage{tikz-cd}
\usepackage{pdfcomment}
\pdfcommentsetup{color=yellow, icon=Note, disable=false, date={.}}

\setlist[enumerate,1]{label=(\arabic*), ref=(\arabic*)}
\setlist[enumerate,3]{label=(\roman*), ref=(\roman*)}

\numberwithin{equation}{section}

\theoremstyle{plain}
\newtheorem{thm}{Theorem}[section]
\newtheorem*{thm*}{Theorem}

\newtheorem{lem}[thm]{Lemma}

\newtheorem*{lem*}{Lemma}

\newtheorem{prop}[thm]{Proposition}

\newtheorem*{prop*}{Proposition}

\newtheorem{cor}[thm]{Corollary}

\newtheorem*{cor*}{Corollary}

\newtheorem*{claim*}{Claim}

\newtheorem*{conj*}{Conjecture}

\theoremstyle{definition}
\newtheorem{defn}[thm]{Definition}
\newtheorem*{defn*}{Definition}

\newtheorem{ques}[thm]{Question}
\newtheorem*{ques*}{Question}

\newtheorem{exa}[thm]{Example}

\newtheorem*{exa*}{Example}

\newtheorem*{acknowledgements}{Acknowledgements}

\theoremstyle{remark}
\newtheorem{rmk}[thm]{Remark}

\newtheorem*{rmk*}{Remark}

\newtheorem*{notation}{Notation}

\numberwithin{figure}{section}
\numberwithin{table}{section}
\numberwithin{equation}{section}

\def \AA {\mathbb{A}}

\def \CC {\mathbb{C}}

\def \PP {\mathbb{P}}

\def \Ecal {\mathcal{E}}

\def \Hcal {\mathcal{H}}
\def \Ical {\mathcal{I}}
\def \Jcal {\mathcal{J}}

\def \Lcal {\mathcal{L}}

\def \Ocal {\mathcal{O}}

\def \Ucal {\mathcal{U}}

\def \Sfr {\mathfrak{S}}

\def \mfr {\mathfrak{m}}

\def \hbar {\bar{h}}

\def \Pbf {\mathbf{P}}

\DeclareMathOperator{\im}{im}

\DeclareMathOperator{\pr}{pr}

\DeclareMathOperator{\Pic}{Pic}

\DeclareMathOperator{\rank}{rank}

\DeclareMathOperator{\supp}{supp}

\DeclareMathOperator{\Sym}{Sym}

\DeclareMathOperator{\res}{res}

\DeclareMathOperator{\codim}{codim}

\DeclareMathOperator{\SL}{SL}

\DeclareMathOperator{\Hom}{Hom}

\DeclareMathOperator{\Gr}{Gr}

\title{The secant variety and syzygies of the Hilbert scheme of two points}
\author{Chiwon Yoon}
\address{Department of Mathematical Sciences, KAIST, 291 Daehak-ro, Yuseong-gu, Deajeon, 34141, Republic of Korea
}
\email{dbs7985@kaist.ac.kr}
\author{Haesong Seo}
\address{Department of Mathematical Sciences, KAIST, 291 Daehak-ro, Yuseong-gu, Deajeon, 34141, Republic of Korea
}
\email{hss21@kaist.ac.kr}

\subjclass[2020]{
    14C05, 
    14N07, 
    13D02, 
    14E05. 
}
\keywords{Hilbert scheme of points, secant variety, syzygy}

\newcommand{\Sec}{\mathrm{Sec}}

\begin{document}

\maketitle

\begin{abstract}
    In this paper, we prove that $\mathrm{Sec} (X^{[2]})$ features the identifiability under the Grothendieck-Pl\"ucker embedding $X^{[2]} \hookrightarrow \PP^N$ when $X$ is embedded by a $4$-very ample line bundle.
    We also prove that the embedding $X^{[2]} \hookrightarrow \PP^N$ satisfies Green's condition $(N_p)$ when the embedding of $X$ is positive enough.
    Accordingly, the singular locus of $\mathrm{Sec} (X^{[2]})$ is exactly $X^{[2]}$ when the embedding of $X$ is positive enough.
    As an application, we describe the geometry of a resolution of singularities from the secant bundle to $\mathrm{Sec}(X^{[2]})$ when $X$ is a surface.
\end{abstract}


\section{Introduction} \label{sec:intro}

Let $X \subset \PP^N$ be a nondegenerate projective variety defined over $\CC$.
The \textit{secant variety} $\Sec(X)$ of $X$ is the closure of the union of secant lines, i.e., the line determined by two points of $X$.

The singular loci of secant varieties are of particular interest.
Terracini's lemma asserts that $\Sec(X)$ is singular along $X$ if $\Sec(X)\neq \PP^N$, but there are only a few cases where singular loci are fully identified.
Micha\l ek, Oeding and Zwiernik \cite{MOZ2015} studied the secant varieties of Segre varieties, analyzing their singular loci and singularity types.
For Veronese varieties, Kanev \cite{Kanev1999} determined the singular loci of secant varieties.
Similar analyses have been conducted for Grassmannians by Galgano and Staffolani \cite{GS2022} and by Manivel and Micha\l ek \cite{MM2015}.
The readers might refer to \cite{AB2013}, \cite{FH2021} and \cite{Han2018} for higher secants.

In a general context, it is believed that the secant variety $\mathrm{Sec}(X)$ exhibits improved behavior of singularity when the embedding of $X$ is sufficiently positive.
Ullery \cite{Ullery2016} confirmed the normality of $\mathrm{Sec}(X)$ when embedded by an adjoint type line bundle.
Further, Chou and Song \cite{CS2018} determined the type of singularities of $\mathrm{Sec}(X)$ under some mild conditions on $X$.
For curves, Ein, Niu and Park \cite{ENP2020} provided an in-depth analysis of the types of singularities of all higher secants.

As seen earlier, numerous studies have focused on the secant varieties of varieties equipped with highly symmetric structures or sufficiently positive embeddings.
Now, we want to shift our attention to the role of identifiability to analyze the singular loci.
A point of the secant variety is called \textit{identifiable} if it lies on the unique secant or tangent line to $X$.
As featured in \cite{GS2022}, the singular locus of the secant variety of the Grassmannian is exactly the non-identifiable locus.
Although this phenomenon is no longer true in general, we see that the identifiability is the first hurdle to showing the smoothness.

On the other hand, in the context of syzygies, several significant results have been established.
A fundamental result of Green \cite{Green1984a, Green1984b} states that if a smooth projective curve $C \subset \PP^n$ has a sufficiently large degree, then the embedding has the property $(N_p)$.
Ein and Lazarsfeld \cite{EL1993} generalized this result to an arbitrary smooth projective variety by showing that an adjoint type line bundle satisfies $(N_p)$.
Moreover, Gallego and Purnaprajna \cite{GP2000} studied syzygies of K3 surfaces and Fano varieties of dimension $n$ and index $n-2$; Park \cite{Park2007} studied syzygies of projective bundles of sufficiently positive vector bundles.

In most cases, there is a natural choice of resolution of singularities of secant varieties --- the secant bundle $\PP(\Ecal_{\Lcal})$ (see Section~\ref{sec:application}).
Vermeire \cite{Vermeire2001} proved that, together with a condition on syzygies, the blowup $\mathrm{Bl}_X \mathrm{Sec}(X) \rightarrow \mathrm{Sec}(X)$ is isomorphic to the natural map $u:\PP(\Ecal_{\Lcal}) \rightarrow \mathrm{Sec}(X)$.
However, this result might not hold without imposing conditions on syzygies (cf. Example~\ref{exa:Imposing-conditions-on-syzygies-needed}), which indicates the essential role of syzygies in describing the map $u$.

The Hilbert scheme $X^{[d]}$ of $d$ points on a smooth projective variety $X$ is a very interesting geometric object.
Note that there is a canonical map $\varphi_{d-1,\Lcal}: X^{[d]} \hookrightarrow \Gr(d,V)$ called the \textit{Grothendieck-Pl\"ucker embedding}, where $\Lcal$ is a $d$-very ample line bundle on $X$, i.e., it separates length $d+1$ subschemes of $X$ (cf.~\cite[Main~Theorem]{CG1990}).
It is natural to ask under what conditions certain algebraic or geometric properties of a polarized variety $(X, \Lcal)$ are preserved in $(X^{[d]}, \varphi_{d-1,\Lcal}^{\ast} \Ocal(1))$.
From this viewpoint, we consider the singular locus of $\mathrm{Sec}(X^{[2]})$ and the syzygies of $X^{[2]}$ under $\varphi_{1,\Lcal}$.

In this study, we prove that the Hilbert scheme $X^{[2]}$ of two points on $X \subset \PP(V)$ features the identifiability and satisfies $(N_p)$ if the embedding of $X$ is positive enough.
Our first main theorem is presented as follows:

\begin{thm} \label{thm:main-theorem-1}
    Let $X$ be a smooth projective variety, and let $\Lcal$ be a $4$-very ample line bundle on $X$.
    Under the embedding $\varphi_{1, \Lcal}:X^{[2]} \hookrightarrow \PP^N$, the non-identifiable locus of $\mathrm{Sec} (X^{[2]})$ is exactly $X^{[2]}$.
\end{thm}

According to our main result, although the embedding of the Hilbert scheme of points cannot even be $2$-very ample in general (cf. Lemma~\ref{lem:lines-in-Hilbert-scheme}), it still exhibits the identifiability.

In the proof, we show that secant lines and tangent lines to $X^{[2]}$ do not intersect out of $X^{[2]}$.
The results by Galgano and Staffolani \cite{GS2022} concerning the secant variety of the Grassmannian play an important role in simplifying our analysis.

The second main theorem deals with the syzygies of the Hilbert scheme of two points.

\begin{thm} \label{thm:main-theorem-2}
    Let $X$ be a smooth projective variety.
    For an integer $p \geq 0$, the Grothendieck-Pl\"ucker embedding $\varphi_{1,\Lcal}: X^{[2]} \hookrightarrow \PP^N$ satisfies $(N_p)$ if $\Lcal$ is a sufficiently positive line bundle on $X$.
\end{thm}

As a corollary, combining with Theorem~\ref{thm:main-theorem-1}, one can argue about the singular locus of the secant variety of $X^{[2]}$.

\begin{cor} \label{cor:main-theorem-3}
    Let $X$ be a smooth projective variety, and let $\Lcal$ be a sufficiently positive line bundle on $X$.
    Under the embedding $\varphi_{1,\Lcal}:X^{[2]} \hookrightarrow \PP^N$, the singular locus of $\mathrm{Sec}(X^{[2]})$ is exactly $X^{[2]}$.
\end{cor}

Note that we do not have a bound for positivity, as it relies on the vanishing of the cohomology groups of certain negatively twisted vector bundles.

To outline the proof of Theorem~\ref{thm:main-theorem-2}, the divisor $B \subset X^{[2]}$ parametrizing nonreduced subschemes is isomorphic to the projectivized cotangent bundle $\PP(\Omega_X)$, so one can apply \cite{Park2007} to verify $(N_p)$ for $B$.
Then we utilize the idea from \cite[Observation~1.3]{GP2000} that the syzygies on a divisor give some information on the syzygies of the ambient space, allowing us to reach the desired result.

As an application, we describe the geometry of $u:\PP(\Ecal_{\Lcal'}) \rightarrow \Sigma = \mathrm{Sec}(X^{[2]})$, where $\Lcal'$ is the line bundle defining the Grothendieck-Pl\"ucker embedding $X^{[2]} \subset \PP^N$.
The problem becomes obvious for curves, so we focus on the surface case.

\begin{prop} \label{prop:application-theorem}
    Let $X$ be a surface.
    The map $u$ factors through $\widetilde{u}: \PP(\Ecal_{\Lcal'}) \rightarrow \widetilde{\Sigma} = \mathrm{Bl}_{X^{[2]}} \Sigma$, and the fiber of $\widetilde{u}$ can be described as follows:
    \begin{enumerate}
        \item if $x \in \widetilde{\Sigma}$ is mapped to $B \subset X^{[2]}$ via the blowup morphism $\widetilde{\Sigma} \rightarrow \Sigma$, then the fiber $\widetilde{u}^{-1}(x)$ is isomorphic to $\PP^2$;
        \item otherwise, the fiber is a reduced point.
    \end{enumerate}
\end{prop}

Even if $X$ contains a line, this result suggests the general idea that the map $u$ is easy to portray when equipped with $(N_2)$.
Note that due to technical difficulties, we cannot extend this result to higher dimensional case.

This paper is organized as follows.
Section~\ref{sec:preliminaries} provides the identifiability result of Grassmannians, along with backgrounds on the Hilbert scheme of points and syzygies of algebraic varieties.
In Section~\ref{sec:identifiability}, we focus on determining the non-identifiable locus of $\mathrm{Sec}(X^{[2]})$.
We give a complete description of the lines in $X^{[d]}$ for $d \geq 2$, and find the conditions on $\dim X$ and $d$ under which a pair of intersecting secant lines or tangent lines to $X^{[d]}$ may exist.
In Section~\ref{sec:syzygies}, we prove the second main theorem by adapting a similar approach to \cite[Theorem~1]{EL1993}.
In Section~\ref{sec:application}, we establish Corollary~\ref{cor:main-theorem-3} and Proposition~\ref{prop:application-theorem} following the directions of \cite{Vermeire2001}.
Finally, we pose several questions on the Hilbert scheme of points in Section~\ref{sec:questions}.

\begin{notation}
For a vector space $V$ of dimension $n$ over $\CC$, the projective space $\Pbf(V)$ (resp. $\PP(V)$) parametrizes one-dimensional subspaces (resp. quotients) of $V$.
For $1 \leq k < n$, the Grassmannian $\Gr(k,V)$ parametrizes $k$-dimensional subspaces in $V$.
For general notations, we refer to \cite{Hartshorne1977}.
\end{notation}

\begin{acknowledgements}
This research is supported by the Institute for Basic Science (IBS-R032-D1).
We would like to express our gratitude to Professor Yongnam Lee for his suggestions on research topics and valuable comments.
We would also like to thank Doyoung Choi for his valuable comments.
We are very grateful to Professor Daniele Agostini and Jinhyung Park for correcting Theorem 4.2 on the smoothness of identifiable locus in the previous version of our manuscript.
\end{acknowledgements}


\section{Preliminaries} \label{sec:preliminaries}

\subsection{Identifiability of the Grassmannian.} \label{subsec:identifiability-of-grassmannians}
Throughout this paper, we work over $\CC$.
Let us recall the definition of identifiability.
For a nondegenerate algebraic variety $X \subset \PP^N$, a point $p$ in $\mathrm{Sec}(X)$ is called \textit{identifiable} if $p$ lies on a line in $\PP^N$ determined by a unique length $2$ subscheme of $X$.
The set of (non-)identifiable points is called the \textit{(non-)identifiable locus}.

Let $V$ be a finite dimensional vector space over $\CC$.
Consider the natural action of $\SL(V)$ on the Grassmannian $\Gr(k,V)$, where $1 \leq k \leq \frac{\dim V}{2}$.
The action extends to $\Pbf(\bigwedge^{k} V)$ via the Pl\"ucker embedding.
In \cite{GS2022}, Galgano and Staffolani classified the $\SL(V)$-orbit stratification of the secant variety of $\Gr(k,V)$ as follows:
 
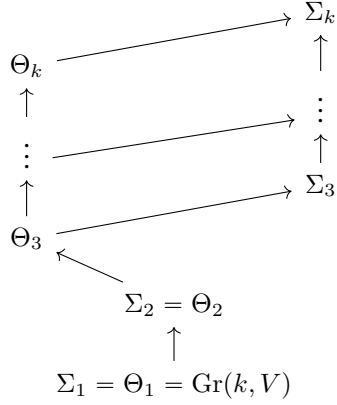
\[
  \begin{tikzcd}
      &[-6ex] &[-6ex] \Sigma_k \\[-3ex]
      \Theta_k \arrow{rru} & & \\[-5ex]
      & & \,\,\vdots\,\, \arrow{uu} \\[-5ex]
      \,\,\vdots\,\, \arrow{uu} \arrow{rru} & & \\[-5ex]
      & & \Sigma_3 \arrow{uu} \\[-3ex]
      \Theta_3 \arrow{uu} \arrow{rru} & & \\[-2ex]
      & \Sigma_2 = \Theta_2 \arrow{lu} & \\[-1ex]
      & \Sigma_1 = \Theta_1 = \Gr(k,V) \arrow{u} &
  \end{tikzcd}
 \]
where the arrow means the inclusion of an orbit into the closure of the other one.

To describe those orbits, they used the following notion of distance:
the \textit{Hamming distance} $\mathrm{dist}(p,q)$ between two points $p,q \in \Gr(k,V)$ is the minimal length of a sequence of lines $\ell_0, \dots, \ell_r$ in $\Gr(k,V)$ such that $p \in \ell_0$, $q \in \ell_r$ and $\ell_i \cap \ell_{i+1} \neq \varnothing$ for $0 \leq i < r$.
If we write $p = [W_1]$ and $q = [W_2]$ for some subspaces $W_1,W_2 \subset V$ of dimension $k$, we obtain
 \[
  \mathrm{dist}(p,q) = \codim_{W_1} (W_1 \cap W_2).
 \]
In particular, if $\mathrm{dist}(p,q)=d$, there are linearly independent vectors $v_1, \dots, v_{k+d}$ in $V$ such that
 \[
  p = [v_1 \wedge \cdots \wedge v_k], \qquad q = [v_1 \wedge \cdots \wedge v_{k-d} \wedge v_{k+1} \wedge \cdots \wedge v_{k+d}]
 \]
as elements of $\Pbf(\bigwedge^{k} V)$.
The orbit $\Sigma_d$ can be given by
 \[
  \Sigma_d = \bigcup_{\mathrm{dist}(p,q) = d} \langle p,q \rangle \setminus \{ p,q \}.
 \]
Similarly, $\Theta_d$ is the union of tangent lines of rank $d$, excluding the tangent points.
Here, a tangent line to $\Gr(k,V)$ at a point $[U]$ is defined by a tangent vector in $T_{[U]}\Gr(k,V) \simeq \Hom_{\CC}(U, V/U)$, and its rank coincides with the matrix rank.

\begin{thm} [{\cite[Main~Result]{GS2022}}] \label{thm:identifiable-locus-for-Grassmannian}
    Let $V$ be a vector space of dimension $n$ over $\CC$.
    For $3 \leq k \leq \frac{n}{2}$ and $d \leq \frac{n}{2}$, The orbit $\Sigma_{d, \Gr(k,V)}$ in $\mathrm{Sec}(\Gr(k,V))$ is identifiable if and only if $d \geq 3$.
    The orbit $\Theta_{d, \Gr(k,V)}$ in $\mathrm{Sec}(\Gr(k,V))$ is identifiable if and only if $d \geq 3$.
\end{thm}

\subsection{The Hilbert scheme of points.} \label{subsec:Hilbert-scheme-of-points}
Let $X$ be a smooth projective variety.
Denote by $X^{[d]} = \mathrm{Hilb}^d(X)$ the \textit{Hilbert scheme of $d$ points} on $X$.

\begin{defn} \label{defn:d-very-ample}
    A line bundle $\Lcal$ on a complete algebraic variety $X$ is \textit{$d$-very ample} if the global sections of $\Lcal$ separate any length $d+1$ subschemes on $X$, i.e., for any length $d+1$ subscheme $Z$ on $X$, we have the surjection
     \[
      H^0(X,\Lcal) \rightarrow H^0(X, \Lcal \otimes \Ocal_Z).
     \]
\end{defn}

\begin{exa} \leavevmode
    \begin{enumerate} [label = (\alph*)]
        \item A line bundle $\Lcal$ is $0$-very ample if and only if it is globally generated; and $\Lcal$ is $1$-very ample if and only if it is very ample.
        \item The tensor product of $k$-very ample and $\ell$-very ample line bundles is $(k+\ell)$-very ample by \cite[Theorem~1.1]{HTT2005}.
    \end{enumerate}
\end{exa}

Suppose that $X$ is equipped with a $(d-1)$-very ample line bundle $\Lcal$.
Then $\Lcal$ determines a morphism $\varphi_{d-1, \Lcal}:X^{[d]} \rightarrow \Gr(d,H^0(X,\Lcal)^{\vee})$ given by
 \[
  \varphi_{d-1, \Lcal}([Z]) = H^0(X, \Lcal \otimes \Ical_Z)
 \]
where $Z \subset X$ is a length $d$ subscheme of $X$ and $\Ical_Z \subset \Ocal_X$ is the ideal sheaf defining $Z$.

\begin{thm} [{\cite[Main Theorem]{CG1990}}] \label{thm:embedding-into-Grassmannian}
    The map $\varphi_{d-1, \Lcal}$ is an embedding if and only if $\Lcal$ is $d$-very ample.
\end{thm}

This embedding is known as the \textit{Grothendieck-Pl\"ucker embedding}.
It is clear that $\varphi_{d-1, \Lcal}$ is nondegenerate.
For later uses, we identify $[Z] \in X^{[d]}$ with $[H^0(X, \Lcal \otimes \Ical_Z)]$, and identify $[W]$ with $[v_1 \wedge \cdots \wedge v_d]$ for a codimension $d$ subspace $W \subset H^0(X, \Lcal)$ and $v_1, \ldots , v_d \in H^0(X,\Lcal)^{\vee}$ defining $W$.

\subsection{Syzygies and Koszul cohomology groups.} \label{subsec:syzygies-and-koszul-cohomologies}
Let $X$ be a smooth projective variety, and let $\Lcal$ be a globally generated line bundle.
We often evaluate the simplicity of $\Lcal$ in terms of its syzygies.
Denote by $S = \Sym^{\bullet} H^0(X, \Lcal)$ the homogeneous polynomial ring.
Define the \textit{section ring} $R = R(X,\Lcal)$ as
 \[
  R = \bigoplus_{k=0}^{n} H^0(X, \Lcal^{k}),
 \]
then it admits a graded $S$-algebra structure.
Recall that $\Lcal$ satisfies $(N_p)$ if the first $p$ terms of the minimal resolution of $R$ are
 \[
  \begin{tikzcd}
      \oplus S(-p-1) \arrow{r} & \cdots \arrow{r} & \oplus S(-2) \arrow{r} & S \arrow{r} & R \arrow{r} & 0.
  \end{tikzcd}
 \]

We present a cohomological criterion, proposed by Ein-Lazarsfeld \cite{EL1993}, that determines whether a given line bundle satisfies $(N_p)$.
Define the \textit{syzygy bundle} $M_{\Lcal}$ of $\Lcal$ by the exact sequence
 \[
  \begin{tikzcd}
      0 \arrow{r} & M_{\Lcal} \arrow{r} & H^0(X, \Lcal) \otimes \Ocal_X \arrow[r, "\mathrm{ev}"] & \Lcal \arrow{r} & 0
  \end{tikzcd}
 \]
where $\mathrm{ev}$ is the evaluation map.

\begin{lem}[{\cite[Lemma~1.6]{EL1993}}] \label{lem:cohomological-criterion-for-np}
    A very ample line bundle $\Lcal$ satisfies $(N_p)$ if
     \begin{equation}
      H^1 \left( X, \bigwedge^q M_{\Lcal} \otimes \Lcal^{k} \right) = 0
     \end{equation}
    for all $q \leq p+1$ and $k \geq 1$.
    The converse also holds if $H^1(X, \Lcal^{k}) = 0$ for $k \geq 1$.
\end{lem}

\noindent The following lemma will be useful for establishing the vanishing above:

\begin{lem}[{\cite[Proposition~2.4]{EL1993}}] \label{lem:non-exact-resolution}
    Assume that $\Lcal$ is $0$-regular with respect to a very ample line bundle $A$, in the sense of Castelnuovo-Mumford.
    Then there exist finite dimensional vector spaces $V_i$ and a complex
     \[
      \begin{tikzcd}
          \cdots \arrow{r} & R_2 = V_2 \otimes A^{-2} \arrow{r} & R_1 = V_1 \otimes A^{-1} \arrow{r} & R_0 = M_{\Lcal} \arrow{r} & 0
      \end{tikzcd}
     \]
    such that $\Hcal_0(R_{\bullet}) = 0$; and $\Hcal_i(R_{\bullet}) = \bigwedge^i N^{\vee} \otimes \Lcal$ for $i \geq 1$, where $N = N_{X/\PP (H^0(A))}$ is the normal bundle of $X \subset \PP (H^0(A))$.
\end{lem}


\section{Identifiability of the Hilbert scheme of points} \label{sec:identifiability}

Let $X$ be a smooth projective variety of dimension $n$, and let $\Lcal$ be a $d$-very ample line bundle on $X$.
We primarily focus on the properties of the Grothendieck-Pl\"ucker embedding $\varphi_{d-1, \Lcal}:X^{[d]} \hookrightarrow \Gr(d, H^0(\Lcal)^{\vee}) \subset \PP^N$.
Even though the Hilbert scheme $X^{[d]}$ contains a line if $\dim X \geq 2$ and $d \geq 2$ by Lemma~\ref{lem:lines-in-Hilbert-scheme}, we obtain the following identifiability result:

\begin{thm} \label{thm:identifiable-locus-for-Hilbert-scheme-of-points}
    Let $X$ be a smooth projective variety, and let $\Lcal$ be a $4$-very ample line bundle on $X$.
    Under the embedding $\varphi_{1, \Lcal}:X^{[2]} \hookrightarrow \PP^N$, the non-identifiable locus of $\mathrm{Sec} (X^{[2]})$ is exactly $X^{[2]}$.
\end{thm}

The assumption of $4$-very ampleness in Theorem~\ref{thm:identifiable-locus-for-Hilbert-scheme-of-points} is necessary, as seen in the following example:

\begin{exa} \label{exa:higher-very-ampleness-is-necessary}
    Suppose that $\Lcal$ is $3$-very ample and there are five points $p_1, \dots, p_5 \in X$ that lie on a $3$-space.
    Denote by $v_i$ the linear equation $H^0(\Lcal) \rightarrow H^0(\Lcal \otimes \Ocal_{p_i}) \simeq \CC$.
    Then there is a nontrivial relation among them, say
     \[
      v_5 = a_1v_1 + \cdots + a_4v_4
     \]
    for some $a_1, \dots, a_4 \in \CC$.
    Since $\Lcal$ is $3$-very ample, none of them are zero.
    Let $Z_i = \{ p_i, p_5 \}$ be a length $2$ subscheme on $X$ for $1 \leq i \leq 4$.
    Then we have
     \[
      [(a_1v_1 + a_2v_2) \wedge v_5] = [(a_3v_3+a_4v_4) \wedge v_5] \in \langle [Z_1], [Z_2] \rangle \cap \langle [Z_3], [Z_4] \rangle,
     \]
    and thus the two secant lines of $X^{[2]}$ intersect.
    If the intersection point is in $X^{[2]}$, say $[Z] \in X^{[2]}$, then we have
     \[
      H^0(\Lcal \otimes \Ical_Z) \cap H^0(\Lcal \otimes \Ical_{Z_1}) \subset H^0(\Lcal \otimes \Ical_{Z_2})
     \]
    and
     \[
      H^0(\Lcal \otimes \Ical_Z) \cap H^0(\Lcal \otimes \Ical_{Z_2}) \subset H^0(\Lcal \otimes \Ical_{Z_1}).
     \]
    By $3$-very ampleness of $\Lcal$, the support of $Z$ contains $p_1$ and $p_2$, i.e., $Z = \{ p_1,p_2 \}$.
    This is a contradiction because the sections of $\Lcal$ vanishing along $Z$ must vanish along $p_5$ as well.
    Hence the two secant lines intersect out of $X^{[2]}$.
\end{exa}

Although our first main theorem is about $X^{[2]}$, we work on $X^{[d]}$ in a general context to analyze the non-identifiable locus of $\mathrm{Sec}(X^{[d]})$.
Throughout this section, we consider the embedding of $X^{[d]}$ described in Section~\ref{subsec:Hilbert-scheme-of-points}.
Further, we assume that $X^{[d]}$ is smooth, i.e., either $\dim X \leq 2$ or $d \leq 3$ (for the proof, see \cite{Lehn2004}).

\subsection{Technical theorems.} \label{subsec:Technical-theorems}

We collect some materials for proving Theorem~\ref{thm:identifiable-locus-for-Hilbert-scheme-of-points}.
Lemma~\ref{lem:sections-inclusion-implies-inclusion}-\ref{lem:Hamming-distance-and-length} indicate that the operations of ideals are compatible with those of $H^0$, provided that the embedding line bundle $\Lcal$ is sufficiently positive.

\begin{lem} \label{lem:sections-inclusion-implies-inclusion}
    Let $\Lcal$ be a $d$-very ample line bundle on $X$.
    Let $Z,W$ be $0$-cycles of length at most $d$ on $X$.
    If $H^0(\Lcal \otimes \Ical_W) \subset H^0(\Lcal \otimes \Ical_Z)$, then $Z \subset W$ holds.
\end{lem}

\begin{proof}
    Suppose not.
    Choose an ideal $\Ical$ maximal among ideals $\Ical_Z \cap \Ical_W \subset \Ical \subsetneq \Ical_W$.
    Then $\Ical$ has colength at most $d+1$ and
    \[
        H^0(\Lcal \otimes \Ical)
        \subsetneq H^0(\Lcal \otimes \Ical_W)
        = H^0(\Lcal \otimes \Ical_{Z}) \cap H^0(\Lcal \otimes \Ical_W)
        \subset H^0(\Lcal \otimes \Ical)
    \]
    by $d$-very ampleness, which is a contradiction.
\end{proof}

\begin{lem} \label{lem:sum-of-ideals}
    Let $\Lcal$ be a $d$-very ample line bundle on $X$ and let $Z,W$ be $0$-cycles on $X$.
    If the colength of $\Ical_Z \cap \Ical_W$ is at most $d+1$, then 
     \begin{equation}
        H^0(\Lcal \otimes (\Ical_{Z}+\Ical_{W})) = H^0(\Lcal \otimes \Ical_{Z}) + H^0(\Lcal \otimes \Ical_{W})
     \end{equation}
    in $H^0(\Lcal)$.
\end{lem}

\begin{proof}
    By $d$-very ampleness, we have
     \begin{align*}
      &\codim H^0(\Lcal \otimes (\Ical_Z + \Ical_W)) + \codim H^0(\Lcal \otimes (\Ical_Z \cap \Ical_W)) \\
      &= \codim H^0(\Lcal \otimes \Ical_Z) + \codim H^0(\Lcal \otimes \Ical_W)
     \end{align*}
    where the codimension is taken in $H^0(\Lcal)$.
    Since
     \begin{equation}
      H^0(\Lcal \otimes (\Ical_Z \cap \Ical_W)) = H^0(\Lcal \otimes \Ical_Z) \cap H^0(\Lcal \otimes \Ical_W),
     \end{equation}
    we have
     \[
      H^0(\Lcal \otimes (\Ical_Z + \Ical_W)) = H^0(\Lcal \otimes \Ical_Z) + H^0(\Lcal \otimes \Ical_W). \qedhere
     \]
\end{proof}

\begin{lem} \label{lem:Hamming-distance-and-length}
    Let $0 \leq d' \leq d$ be an integer.
    Let $\Lcal$ be a $\min \{ d+d', 2d-1 \}$-very ample line bundle on $X$, and let $Z_1, Z_2 \subset X$ be $0$-cycles of length $d$ whose Hamming distance in $X^{[d]}$ is $d'$.
    Then $\Ical_{Z_1} \cap \Ical_{Z_2}$ has colength $d+d'$ and $\Ical_{Z_1} + \Ical_{Z_2}$ has colength $d-d'$.
\end{lem}

\begin{proof}
    It suffices to check for $d' = d$.
    Since the colength of $\Ical_{Z_1} \cap \Ical_{Z_2}$ cannot exceed $2d$, the statement is true by $(2d-1)$-very ampleness.
\end{proof}

We determine the lines in the Hilbert scheme $X^{[d]}$ of $d$ points.
First, we describe lines in the Grassmannian without proof.

\begin{prop} \label{prop:lines-in-Grassmannian}
    Let $V$ be a vector space of dimension $n$ over $\CC$ and let $1 \leq k \leq \frac{n}{2}$.
    Under the Pl\"ucker embedding $\Gr(k,V) \hookrightarrow \PP^N$, a line $\ell \subset \Gr(k,V)$ corresponds to a pair of subspaces $W_0 \subset W_1 \subset V$ of dimension $k-1$ and $k+1$, respectively, and a point $[U] \in \ell$ corresponds to a $k$-dimensional subspace $W_0 \subset U \subset W_1$.
\end{prop}

\begin{lem} \label{lem:lines-in-Hilbert-scheme}
    Assume that $\Lcal$ is a $(d+1)$-very ample line bundle on $X$.
    For given distinct two points $[Z_1], [Z_2] \in X^{[d]}$, the line $\langle [Z_1],[Z_2] \rangle$ is contained in $X^{[d]}$ if and only if $\mathrm{dist}([Z_1],[Z_2]) = 1$ and $Z_1$ differs from $Z_2$ only at one point of $X$, i.e., $\Ocal_{Z_1, p} = \Ocal_{Z_2, p}$ except for a single point $p \in X$.
    Moreover, if a line meets $X^{[d]}$ in at least three distinct points, then it is contained in $X^{[d]}$.
\end{lem}

\begin{proof}
    Suppose that $\langle [Z_1], [Z_2] \rangle \subset X^{[d]}$.
    Since the line is contained in the Grassmannian, the Hamming distance between $[Z_1]$ and $[Z_2]$ must be $1$.
    Let $Z$ and $W$ be the subschemes defined by $\Ical_{Z_1} + \Ical_{Z_2}$ and $\Ical_{Z_1} \cap \Ical_{Z_2}$, respectively.
    According to Lemma~\ref{lem:Hamming-distance-and-length}, $Z$ and $W$ have length $d-1$ and $d+1$, respectively.
    
    By Proposition~\ref{prop:lines-in-Grassmannian} and Lemma~\ref{lem:sections-inclusion-implies-inclusion}, points on the line $\langle [Z_1], [Z_2] \rangle$ correspond to length $d$ subschemes between $Z$ and $W$.
    If $\res(Z,Z_1) \neq \res(Z,Z_2)$, where $\mathrm{res}$ denotes the residue, then those should be either $Z_1$ or $Z_2$.
    This is not the case.
    Hence $Z_1$ and $Z_2$ differ only at one point of $X$.

    For the converse statement, consider $Z$ and $W$ as defined earlier and $p$ as the only point at which $Z_1$ and $Z_2$ differ.
    Then we have $n \geq 2$ a priori.
    Let $R = \hat{\Ocal}_{X,p}$ be the completion of $\Ocal_{X,p}$ with respect to the maximal ideal $\mfr \subset \Ocal_{X,p}$.
    By \cite[Lemma~1.3.2.(1)]{Gottsche1994}, there exist integers $j_1 \geq j_2 \geq 1$ and $f_{k} \in (\hat{\Ical}_{Z,p} \cap \hat{\mfr}^{j_k}) \setminus (\hat{\Ical}_{W,p} \cap \hat{\mfr}^{j_k})$ such that
     \[
      \hat{\Ical}_{W,p} \subsetneq
      \hat{\Ical}_{W,p} + (f_{1}) \subsetneq
      \hat{\Ical}_{W,p} + (f_{1}, f_{2}) =
      \hat{\Ical}_{Z,p}.
     \]
    Observe that $\hat{\Ical}_{W,p}+(f_1)=\hat{\Ical}_{W,p}+\CC f_1$ and $\hat{\Ical}_{Z,p}=\hat{\Ical}_{W,p}+\CC f_1+\CC f_2$ as $\CC$-vector spaces.
    
    Suppose that
    \[
        \hat{\Ical}_{W,p} + (a_1f_1+a_2f_2) = \hat{\Ical}_{W,p} + (a_1'f_1+a_2'f_2)
    \]
    for some $[a_1:a_2]\neq [a_1':a_2']$ in $\PP^1$.
    Then one can write
     \[
      f_1=g + (h+c)(af_1+f_2)
     \]
    for some $g \in \hat{\Ical}_{W,p}$, $h \in \hat{\mfr}$ and $a, c \in \CC$.
    If $c = 0$, we have
     \[
      \hat{\Ical}_{W,p}+(f_1) = \hat{\Ical}_{W,p}+(hf_2)
     \]
    and
     \[
      \hat{\Ical}_{W,p}+(bf_1+f_2)=\hat{\Ical}_{W,p}+(f_2)
     \]
    for any $b \in \CC$.
    If $c \neq 0$, then $ac \neq 1$; otherwise, $f_2 \in \hat{\Ical}_{W,p} + (f_1)$.
    As $\hat{\mfr} f_1 \subset \hat{\Ical}_{W,p}$, we have
    \[
        \hat{\Ical}_{W,p} + (bf_1+f_2)
        =\hat{\Ical}_{W,p} +\left((b(h+c) +1-ac)f_2\right) \supsetneq \hat{\Ical}_{W,p}
    \]
    for any $b \in \CC$.
    The ideal equals $\hat{\Ical}_{W,p} +(f_2)$ unless $b=\frac{ac-1}{c}$; or else, it equals $\hat{\Ical}_{W,p} + (hf_2)$.
    In any cases, a nontrivial ideal between $\hat{\Ical}_{Z,p}$ and $\hat{\Ical}_{W,p}$ is either $\hat{\Ical}_{W,p} + (f_2)$ or $\hat{\Ical}_{W,p} + (hf_2)$.
    Only one of them can be nontrivial because they have different colengths in $R$.
    This contradicts the existence of two nontrivial ideals $\hat{\Ical}_{Z_1,p}$ and $\hat{\Ical}_{Z_2,p}$.

    As a consequence, the ideals $\hat{\Ical}_{W,p} + (a_1f_1+a_2f_2)$ are distinct and strictly contained in $\hat{\Ical}_{Z,p}$.
    Therefore, they correspond to the points on the line in $X^{[d]}$.

    The final assertion holds true, as the line $\langle [Z_1],[Z_2] \rangle$ intersects $X^{[d]}$ in at least three points exactly when the specified two conditions are met.
\end{proof}

The geometric meaning of Lemma~\ref{lem:lines-in-Hilbert-scheme} is that lines in $X^{[d]}$ are those in the divisor $B \subset X^{[d]}$ that are exceptional with respect to the Hilbert-Chow morphism $X^{[d]} \rightarrow X^{(d)}$.

If $X$ is a curve, two length $d$ subschemes $Z_1, Z_2 \subset X$ cannot differ only in one point of $X$.
Hence we have the following:

\begin{cor} \label{cor:No-lines-for-curve}
    If $X$ is a curve, then no three points in $X^{[d]}$ are colinear.
\end{cor}

\subsection{Tangent lines to the Hilbert scheme of points.} \label{subsec:tangent-lines}

To analyze tangent lines, we recall the fact from deformation theory (see \cite[Theorem~3.2]{Lehn2004}): the tangent space of $X^{[d]}$ at $[Z]$ is isomorphic to $\Hom_{\Ocal_Z}(\Ical_Z/\Ical_Z^2, \Ocal_Z)$.
In particular, it can be regarded as a matrix.

Let $r \geq 0$ and let $\Lcal$ be a $(d+r-1)$-very ample line bundle on $X$.
Let $t \in T_{X^{[d]},[Z]}$ be of matrix rank $r$.
Then the composition
 \[
  \widetilde{t}:H^0(\Lcal \otimes \Ical_Z) \rightarrow (\Lcal \otimes \Ical_Z)|_Z \simeq \Ical_Z/\Ical_Z^2 \xrightarrow{t} \Ocal_Z
 \]
has rank $r$ as well.
Indeed, write $\ker(t) = \Jcal/\Ical_Z^2$ for some $\Ical_Z^2 \subset \Jcal \subset \Ical_Z$.
Then $\Jcal$ has colength $d+r$, so $H^0(\Lcal \otimes \Jcal)$ has codimension $d+r$ in $H^0(\Lcal)$ by $(d+r-1)$-very ampleness.
Since $\ker(\widetilde{t}) = H^0(\Lcal \otimes \Jcal)$, it follows that $\widetilde{t}$ has rank $r$ as well.

One might view $t$ as a tangent vector to the Grassmannian, i.e., an element of
 \[
  \Hom_{\CC}(H^0(\Lcal \otimes \Ical_Z), H^0(\Lcal)/H^0(\Lcal \otimes \Ical_Z))
    \simeq H^0(\Lcal)/H^0(\Lcal \otimes \Ical_Z) \otimes H^0(\Lcal \otimes \Ical_Z)^{\vee}.
 \]
Thus one can write
 \[
  [Z] = [v_1 \wedge \cdots \wedge v_d]
 \]
and
 \begin{equation}
  t = \overline{v_{d-r+1}^{\ast}} \otimes \overline{v_{d+1}^{\vphantom{\ast}}} + \cdots + \overline{v_d^{\ast}} \otimes \overline{v_{d+r}^{\vphantom{\ast}}}
 \end{equation}
for some basis $v_1, \dots, v_N \in H^0(\Lcal)^{\vee}$.
Here, we denote by $v_1^{\ast}, \dots, v_N^{\ast} \in H^0(\Lcal)$ the dual basis, $\overline{v_i^{\ast}}$ the image of $v_i^{\ast}$ in $H^0(\Lcal)/H^0(\Lcal \otimes \Ical_Z)$, and $\overline{v_i}$ the image of $v_i$ under the quotient map $H^0(\Lcal)^{\vee} \twoheadrightarrow H^0(\Lcal \otimes \Ical_Z)^{\vee}$.
Note that we can restore $\Ical_Z/\Ical_Z^2 \xrightarrow{t} \Ocal_Z$ from this notation: $\widetilde{t}$ sends $v_{d+i}^{\ast}$ to $\overline{v_{d-r+i}^{\ast}}$ for $1 \leq i \leq r$ and $v_{d+i}^{\ast}$ to $0$ for $i > r$.
From now on, we will omit the bar notation if there is no confusion.
Then points (other than $[Z]$ itself) on the tangent line $\langle t \rangle$ can be written as
 \begin{equation} \label{eqn:points-in-tangent-line}
  \begin{split}
   &[v_1 \wedge \cdots \wedge v_{d-r} \wedge (v_{d+1} \wedge v_{d-r+2} \wedge v_{d-r+3} \wedge \cdots \wedge v_{d-1} \wedge v_{d} \\
   &\qquad + v_{d-r+1} \wedge v_{d+2} \wedge v_{d-r+3} \wedge \cdots \wedge v_{d-1} \wedge v_{d} + \cdots \\
   &\qquad + v_{d-r+1} \wedge v_{d-r+2} \wedge v_{d-r+3} \wedge \cdots \wedge v_{d-1} \wedge v_{d+r})] \in \langle t \rangle.
  \end{split}
 \end{equation}
This coincides with the description in \cite{GS2022}.
In summary, one can observe the following:
\begin{itemize}
    \item[\hypertarget{dagger1}{($\dagger_1$)}] The preimage $\pi^{-1}(\im(t))$ under the map $\pi:H^0(\Lcal) \rightarrow H^0(\Lcal)/H^0(\Lcal \otimes \Ical_Z)$ is defined by $v_1, \dots, v_{d-r}$ in $H^0(\Lcal)$.
    \item[\hypertarget{dagger2}{($\dagger_2$)}] The kernel $\ker(\widetilde{t})$ is defined by $v_1, \dots, v_{d+r}$ in $H^0(\Lcal)$.
    \item[\hypertarget{dagger3}{($\dagger_3$)}] $\supp(\Ocal_X/\Jcal) = \supp(Z)$, where $\ker(t) = \Jcal/\Ical_Z^2$ for some $\Ical_Z^2 \subset \Jcal \subset \Ical_Z$.
\end{itemize}

\begin{cor} \label{cor:secant-and-tangent-are-not-the-same}
    Let $\Lcal$ be a $(d+1)$-very ample line bundle on $X$.
    If a secant line and a tangent line to $X^{[d]}$ coincide, then the line is contained in $X^{[d]}$.
\end{cor}

\begin{proof}
    Let $[Z_1],[Z_2] \in X^{[d]}$ be distinct two points and let $t \in T_{X^{[d]}, [Z_1]}$ be a nonzero tangent vector such that $\langle t \rangle = \langle [Z_1], [Z_2] \rangle$.
    Since lines intersecting the Grassmannian in at least three points (with multiplicity) must lie in the Grassmannian, we have $\rank(t) = \mathrm{dist}([Z_1], [Z_2]) = 1$.
    Then $t = t|_p$ for some $p \in \supp(Z_1)$, and thus $Z_1$ differs from $Z_2$ only at $p$.
    By Lemma~\ref{lem:lines-in-Hilbert-scheme}, the line is contained in $X^{[d]}$.
\end{proof}

\subsection{The proof of Theorem~\ref{thm:identifiable-locus-for-Hilbert-scheme-of-points}.} \label{subsec:proof-of-main-theorem-1}

Let $\Lcal$ be a $(d+2)$-very ample line bundle on $X$.
We aim to classify the non-identifiable locus of the secant variety $\mathrm{Sec}(X^{[d]})$ as well as possible.
However, there are obstructions to the identifiability results for some ranges of $n$ and $d$.

\begin{exa} \label{exa:Hilbert-scheme-of-three-points-weird}
    If $n \geq 2$ and $d=3$, choose local coordinates $x,y,z_1, \dots, z_{n-2}$ at a point $p \in X$.
    Define
     \begin{gather*}
      \Ical_{Z_1} = (y,z_1, \dots, z_{n-2}) + \mfr^3, \qquad \Ical_{Z_2} = (x,z_1, \dots, z_{n-2}) + \mfr^3, \\
      \Ical_{Z_3} = (y+x^2,z_1, \dots, z_{n-2}) + \mfr^3, \\
      \Ical_{Z_4} = (x+y^2,z_1, \dots, z_{n-2}) + \mfr^3
     \end{gather*}
    where $\mfr = \mfr_p$ is the maximal ideal at $p$.
    Then the secant lines $\langle [Z_1], [Z_2] \rangle$ and $\langle [Z_3], [Z_4] \rangle$ meet out of $X^{[3]}$.
\end{exa}

In the remaining part of this section, we assume that either $n = 1$ or $d = 2$.
Note that if two (unordered) pairs of distinct points yield the same secant line, then they are contained in $X^{[d]}$ by Lemma~\ref{lem:lines-in-Hilbert-scheme}.
Also, if a secant line and a tangent line coincide, then they are contained in $X^{[d]}$ by Corollary~\ref{cor:secant-and-tangent-are-not-the-same}.
In particular, tangent lines to two distinct points do not coincide unless they are contained in $X^{[d]}$.

\begin{prop} \label{prop:Secant-lines-intersect}
    Two distinct secant lines to $X^{[d]}$ do not intersect out of $X^{[d]}$.
\end{prop}

\begin{proof}
Let $[Z_i] \in X^{[d]}$, $1 \leq i \leq 4$ be four points.
For simplicity, let $\ell_{ij}$ denote the line $\langle [Z_i], [Z_j] \rangle$ and $d_{ij}$ denote the distance $\mathrm{dist}([Z_i], [Z_j])$ for $1 \leq i, j \leq 4$.
Assume that $\ell_{12}$ and $\ell_{34}$ meet outside $X^{[d]}$.
From Section~\ref{subsec:identifiability-of-grassmannians}, we infer that $d_{12} = d_{34} \in \{ 1,2 \}$.

\hypertarget{Case1-1}{\textit{Case 1.}} Assume that 
$d_{12} = d_{34} = 1$.
Let $Z_{12}$ (resp. $Z_{34}$) be the length $d-1$ subscheme defined by $\Ical_{Z_1} + \Ical_{Z_2}$ (resp. $\Ical_{Z_3} + \Ical_{Z_4}$); and let $W_{12}$ (resp. $W_{34}$) be the length $d+1$ subscheme defined by $\Ical_{Z_1} \cap \Ical_{Z_2}$ (resp. $\Ical_{Z_3} \cap \Ical_{Z_4}$).
By Proposition~\ref{prop:lines-in-Grassmannian}, the points $[V] \in \ell_{12}$ correspond to the subspaces $V \subset H^0(\Lcal)$ of codimension $d$ with
 \[
  H^0(\Lcal \otimes \Ical_{W_{12}}) \subset V \subset H^0(\Lcal \otimes \Ical_{Z_{12}}),
 \]
and a similar result holds for $\ell_{34}$.
Let $[V] \in \ell_{12} \cap \ell_{34}$ be the intersection point.
Since the two lines are distinct, we have either $Z_{12} \neq Z_{34}$ or $W_{12} \neq W_{34}$.

If $Z_{12} \neq Z_{34}$, it follows that
 \[
  V = H^0(\Lcal \otimes \Ical_{Z_{12}}) \cap H^0(\Lcal \otimes \Ical_{Z_{34}})
    = H^0(\Lcal \otimes (\Ical_{Z_{12}} \cap \Ical_{Z_{34}})),
 \]
so $[V]$ belongs to $X^{[d]}$.

If $W_{12} \neq W_{34}$, then they have Hamming distance $1$ in $X^{[d+1]}$.
According to Lemma~\ref{lem:Hamming-distance-and-length}, $\Ical_{W_{12}} \cap \Ical_{W_{34}}$ has colength $d+2$.
Then by Lemma~\ref{lem:sum-of-ideals}, 
 \[
  V = H^0(\Lcal \otimes \Ical_{W_{12}}) + H^0(\Lcal \otimes \Ical_{W_{34}}) = H^0(\Lcal \otimes (\Ical_{W_{12}} + \Ical_{W_{34}})).
 \]
Hence we would get a contradiction in any cases.

\hypertarget{Case1-2}{\textit{Case 2.}} Now assume that $d_{12} = d_{34} = 2$.
Write
 \[
  [Z_1] = [v_1 \wedge \cdots \wedge v_{d-2} \wedge v_{d-1} \wedge v_d], \qquad
  [Z_2] = [v_1 \wedge \cdots \wedge v_{d-2} \wedge v_{d+1} \wedge v_{d+2}]
 \]
for some linearly independent vectors $v_1, \dots, v_{d+2} \in H^0(\Lcal)^{\vee}$.
We will apply appropriate linear coordinate changes to $v_i$ as new expressions emerge.
The intersection point $P = \ell_{12} \cap \ell_{34}$ is
 \[
  [v_1 \wedge \cdots \wedge v_{d-2} \wedge (v_{d-1} \wedge v_d + v_{d+1} \wedge v_{d+2})],
 \]
so the annihilator of the kernel of $\wedge P : H^0(\Lcal)^{\vee}\rightarrow \bigwedge^{d+1} H^0(\Lcal)^{\vee}$ is
\[
    H^0(\Lcal \otimes \Ical_{Z_1})+H^0(\Lcal \otimes \Ical_{Z_2})=H^0(\Lcal \otimes \Ical_{Z_3})+H^0(\Lcal \otimes \Ical_{Z_4}).
\]
In particular, $d_{ij} \leq 2$ for any $i,j$.

\hypertarget{Case1-2-a}{\textit{Case 2-a.}} If $d_{13} = d_{23} = 1$, the expression for $[Z_3]$ would be
 \[
  [Z_3] = [v_1 \wedge \cdots \wedge v_{d-2} \wedge v_{d-1} \wedge v_{d+1}].
 \]
Since $[Z_4] \in \langle P, [Z_3] \rangle$, the vector
 \begin{align*}
  &v_1 \wedge \cdots \wedge v_{d-2} \wedge (v_{d-1} \wedge v_d + v_{d+1} \wedge v_{d+2} + av_{d-1} \wedge v_{d+1}) \\
  &= v_1 \wedge \cdots \wedge v_{d-2} \wedge (v_{d-1} \wedge (v_d + av_{d+1}) + v_{d+1} \wedge v_{d+2})
 \end{align*}
should be decomposable for some $0 \neq a \in \CC$.
But since the vectors $v_1, \dots, v_{d-1}$, $v_{d}+av_{d+1}$, $v_{d+1}$ and $v_{d+2}$ are linearly independent, this is not the case.

\hypertarget{Case1-2-b}{\textit{Case 2-b.}} If $d_{13} = 1$ and $d_{23} = 2$, one can write
 \[
  [Z_3] = [v_1 \wedge \cdots \wedge v_{d-2} \wedge v_{d-1} \wedge (v_{d} + v_{d+1})]
 \]
and
 \[
  [Z_4] = [v_1 \wedge \cdots \wedge v_{d-2} \wedge v_{d+1} \wedge (v_{d-1} + v_{d+2})].
 \]
Hence $d_{13} = d_{24} = 1$ and $d_{14} = d_{23} = 2$.
Since the four points $[Z_i]$ are coplanar, $\ell_{13}$ and $\ell_{24}$ must intersect.
They intersect within $X^{[d]}$ by Case~\hyperlink{Case1-1}{1}, and thus they are contained in $X^{[d]}$ by Lemma~\ref{lem:lines-in-Hilbert-scheme}.
In particular, all the $Z_i$ have the same supports.

The case $n = 1$ does not happen by Corollary~\ref{cor:No-lines-for-curve}.

If $d = 2$, we get a contradiction as $\supp(Z_1)$ is disjoint from $\supp(Z_2)$.

\hypertarget{Case1-2-c}{\textit{Case 2-c.}} Finally, suppose that $d_{13} = d_{23} = 2$.
We may assume that $d_{14} = d_{24} = 2$, otherwise we can argue as in Case \hyperlink{Case1-2-a}{2-a} and \hyperlink{Case1-2-b}{2-b}.

If $n=1$, from the equality
 \[
  H^0(\Lcal \otimes (\Ical_{Z_1} \cap \Ical_{Z_2})) = H^0(\Lcal \otimes (\Ical_{Z_3} \cap \Ical_{Z_4})),
 \]
we have $Z_1 \cap Z_2 = Z_1 \cap Z_3$ and $Z_1 \cup Z_2 = Z_1 \cup Z_3$.
Hence we have
 \[
  Z_2 = (Z_1 \cup Z_2) + (Z_1 \cap Z_2) - Z_1 = (Z_1 \cup Z_3) + (Z_1 \cap Z_3) - Z_1 = Z_3
 \]
as divisors on $X$, which is a contradiction.

When $d = 2$, the supports of $Z_i$ are pairwise disjoint for $1 \leq i \leq 4$.
As above, we have $Z_1 \cup Z_2 = Z_3 \cup Z_4$ by $4$-very ampleness, which is impossible.
\end{proof}

\begin{prop} \label{prop:Secant-line-and-tangent-line-intersect}
    A secant line and a tangent line to $X^{[d]}$ do not intersect out of $X^{[d]}$.
\end{prop}

\begin{proof}

Suppose that three distinct points $[Z_1], [Z_2], [Z_3] \in X^{[d]}$ and a tangent vector $t \in T_{X^{[d]}, [Z_1]}$ are given that $\langle [Z_2], [Z_3] \rangle$ and $\langle t \rangle$ meet outside $X^{[d]}$.
Denote by $d_{ij}$ the distance $\mathrm{dist}([Z_i], [Z_j])$ for $1 \leq i, j \leq 3$.
As before, we have $d_{23} = \rank(t) \in \{ 1,2 \}$.

\hypertarget{Case2-1}{\textit{Case 1.}}
Assume that $d_{23} = \rank(t) = 1$.
Let $Z_{23}$ be the length $d-1$ subscheme defined by $\Ical_{Z_2} + \Ical_{Z_3}$, and let $W_{23}$ be the length $d+1$ subscheme defined by $\Ical_{Z_2} \cap \Ical_{Z_3}$.
Let $\im(t) = \Ical/\Ical_{Z_1}$ for some $\Ical_{Z_1} \subset \Ical \subset \Ocal_X$, and let $\ker(t) = \Jcal/\Ical_{Z_1}^2$ for some $\Ical_{Z_1}^2 \subset \Jcal \subset \Ical_{Z_1}$.
Let $Z_0$ be the subscheme of length $d-1$ defined by $\Ical$, and let $W_0$ be the subscheme of length $d+1$ defined by $\Jcal$.
By Proposition~\ref{prop:lines-in-Grassmannian}, the points $[V] \in \langle t \rangle$ correspond to the subspaces $V \subset H^0(\Lcal)$ of codimension $d$ with
 \[
  H^0(\Lcal \otimes \Ical_{W_0}) \subset V \subset H^0(\Lcal \otimes \Ical_{Z_0}).
 \]
A similar argument as in Proposition~\ref{prop:Secant-lines-intersect} would give a contradiction.

\hypertarget{Case2-2}{\textit{Case 2.}} Now assume that $d_{23} = \rank(t) = 2$.
Write
 \[
  [Z_2] = [v_1 \wedge \cdots \wedge v_{d-2} \wedge v_{d-1} \wedge v_d], \qquad
  [Z_3] = [v_1 \wedge \cdots \wedge v_{d-2} \wedge v_{d+1} \wedge v_{d+2}]
 \]
for some linearly independent vectors $v_1, \dots, v_{d+2} \in H^0(\Lcal)^{\vee}$.
By (\hyperlink{dagger1}{$\dagger_1$}), we have $d_{1i} \leq 2$ for $i = 2,3$.

\hypertarget{Case2-2-a}{\textit{Case 2-a.}} If $d_{12} = 1$ and $d_{13} = 2$, one might write
 \[
  [Z_1] = [v_1 \wedge \cdots \wedge v_{d-2} \wedge v_{d-1} \wedge (v_d + v_{d+1})].
 \]
Then by (\hyperlink{dagger1}{$\dagger_1$}) and (\hyperlink{dagger2}{$\dagger_2$}), $t$ can be expressed as
 \[
  t = v_{d-1}^{\ast} \otimes (a(v_d - v_{d+1}) + bv_{d+2})
    + \frac{1}{2} \left( v_{d}^{\ast} - v_{d+1}^{\ast} \right) \otimes (a'(v_d - v_{d+1}) + b'v_{d+2})
 \]
for some $a,b,a',b' \in \CC$.
However, the equality
 \begin{align*}
  &[v_1 \wedge \cdots \wedge v_{d-2} \wedge (v_{d-1} \wedge v_d + v_{d+1} \wedge v_{d+2})] \\
  &= [v_1 \wedge \cdots \wedge v_{d-2} \wedge ((a(v_d - v_{d+1}) + bv_{d+2}) \wedge (v_d + v_{d+1}) \\
  &\qquad + v_{d-1} \wedge (a'(v_d - v_{d+1}) + b'v_{d+2}))]
 \end{align*}
cannot hold, which is easily verified by taking $\wedge (v_d-v_{d+1}) \wedge v_{d+2}$.
Thus this case is abolished.

\hypertarget{Case2-2-b}{\textit{Case 2-b.}} Assume that $d_{12} = d_{13} = 1$.

When $n = 1$, there exist (possibly equal) $p,q \in \supp(Z_1)$ such that
 \[
  Z_2 \cap Z_3 = Z_1 - p - q, \qquad Z_2 \cup Z_3 = Z_1 + p + q
 \]
by (\hyperlink{dagger1}{$\dagger_1$}) and (\hyperlink{dagger2}{$\dagger_2$}).
Hence $p$ and $q$ are distinct and we may assume that $Z_2 = Z_1 + p - q$ and $Z_3 = Z_1 - p + q$.
Thus we may write
 \[
  [Z_1] = [v_1 \wedge \dots \wedge v_{d-2} \wedge v_{d-1} \wedge v_{d+1}].
 \]
Since $t = t|_p \oplus t|_q$, we have
 \[
  t = av_{d-1}^{\ast} \otimes v_{d} + bv_{d+1}^{\ast} \otimes v_{d+2}
 \]
for some $a,b \in \CC$. But then
 \begin{align*}
  &[v_1 \wedge \dots \wedge v_{d-2} \wedge (v_{d-1} \wedge v_d + v_{d+1} \wedge v_{d+2})] \\
  &\neq [v_1 \wedge \dots \wedge v_{d-2} \wedge (av_d \wedge v_{d+1} + bv_{d-1} \wedge v_{d+2})],
 \end{align*}
which yields a contradiction.

When $d = 2$, the subschemes $Z_2$ and $Z_3$ are disjoint.
Since $Z_1 \subset Z_2 \cup Z_3$ by (\hyperlink{dagger3}{$\dagger_3$}) and they have the same support, $Z_1$ should be reduced.
Let $p = Z_1 \cap Z_2$ and $q = Z_1 \cap Z_3$.
Write $[Z_2] = [v_1 \wedge v_2]$, $[Z_3] = [v_3 \wedge v_4]$ and $[Z_1] = [v_1 \wedge v_3]$ for some linearly independent vectors $v_1, v_2, v_3, v_4 \in H^0(\Lcal)^{\vee}$.
Since $Z_1$ is reduced, one can decompose $t = t|_{p} \oplus t|_{q}$.
Note that $\ker t|_{p} = \Ical_{Z_2}/\Ical_p^2$ and $\ker t|_{q} = \Ical_{Z_3}/\Ical_q^2$.
Hence from (\hyperlink{dagger2}{$\dagger_2$}), we infer that
 \[
  t = av_1^{\ast} \otimes v_2 + bv_3^{\ast} \otimes v_4
 \]
for some $a,b \in \CC$.
This leads to a contradiction because
 \[
  [v_1 \wedge v_2 + v_3 \wedge v_4] \neq [bv_1 \wedge v_4 + av_2 \wedge v_3].
 \]

\hypertarget{Case2-2-c}{\textit{Case 2-c.}} Assume that $d_{12} = d_{13} = 2$.
The case $n = 1$ is impossible because $Z_2 = Z_1 + p - q$ and $Z_3 = Z_1 - p + q$ for some distinct $p,q \in X$.
Also, the case $d = 2$ is impossible because $\supp(Z_i)$ is pairwise disjoint for $1 \leq i \leq 3$ but $Z_1 \subset Z_2 \cup Z_3$ by (\hyperlink{dagger3}{$\dagger_3$}).
\end{proof}

\begin{prop} \label{prop:Tangent-lines-intersect}
    Two distinct tangent lines to $X^{[d]}$ do not intersect out of $X^{[d]}$.
\end{prop}

\begin{proof}

Suppose that we are given two distinct points $[Z_1], [Z_2] \in X^{[d]}$ and tangent vectors $t_i \in T_{X^{[d]}, [Z_i]}$ such that the tangent lines $\langle t_1 \rangle$ and $\langle t_2 \rangle$ meet outside $X^{[d]}$.
We have $\rank(t_1) = \rank(t_2) \in \{ 1,2 \}$ and $\supp(Z_1) = \supp(Z_2)$ by (\hyperlink{dagger2}{$\dagger_2$}) and (\hyperlink{dagger3}{$\dagger_3$}).

\hypertarget{Case3-1}{\textit{Case 1.}}
For the case $\rank(t_1) = \rank(t_2) = 1$, arguing as in Proposition~\ref{prop:Secant-lines-intersect} and Proposition~\ref{prop:Secant-line-and-tangent-line-intersect} would lead to a contradiction.

\hypertarget{Case3-2}{\textit{Case 2.}} Suppose that $\rank(t_1) = \rank(t_2) = 2$.
Then we have $\mathrm{dist}([Z_1], [Z_2]) \leq 2$ by (\hyperlink{dagger1}{$\dagger_1$}).

\hypertarget{Case3-2-a}{\textit{Case 2-a.}} Assume that $\mathrm{dist}([Z_1],[Z_2]) = 1$.

When $n = 1$, by (\hyperlink{dagger1}{$\dagger_1$}) and (\hyperlink{dagger2}{$\dagger_2$}), there exist (possibly equal) $p, q, p', q' \in X$ such that
 \[
  Z_1 - p - q = Z_2 - p' - q', \qquad Z_1 + p + q = Z_2 + p' + q'.
 \]
Thus $p+q = p'+q'$ as divisors on $X$, which in turn implies that $Z_1 = Z_2$, a contradiction.

When $d = 2$, observe that $Z_1$ and $Z_2$ are supported in the same point, say $p \in X$.
Let $x,y,z_1, \dots, z_{n-2}$ be local coordinates at $p$ such that
 \[
  \Ical_{Z_1} = (y,z_1, \dots, z_{n-2}) + \mfr^2, \qquad
  \Ical_{Z_2} = (x,z_1, \dots, z_{n-2}) + \mfr^2
 \]
where $\mfr$ is the maximal ideal at $p$.
Say $\ker(t_1) = \Jcal/\Ical_{Z_1}^2$ for some $\Ical_{Z_1}^2 \subset \Jcal \subset \Ical_{Z_1}$.
Then we have $\Ical_{Z_2}^2 \subset \Jcal \subset \Ical_{Z_2}$ and $\ker(t_2) = \Jcal/\Ical_{Z_2}^2$ by (\hyperlink{dagger2}{$\dagger_2$}).
Note that $\Ical_{Z_1}/\Jcal \simeq \Ocal_{Z_1}$ as $\Ocal_{Z_1}$-modules.
Thus $\Jcal \not\supset \mfr^2$ must hold; otherwise, the former is reduced while the latter is not.
It follows that
 \[
  \Jcal = \Ical_{Z_1}^2 + \Ical_{Z_2}^2 + (z_1, \dots, z_{n-2}) = (x^2, y^2, z_1, \dots, z_{n-2}) + \mfr^3.
 \]
Hence one can write
 \[
  [Z_1] = [v_1 \wedge v_2], \qquad
  [Z_2] = [v_1 \wedge v_3]
 \]
and
 \[
  t_1 = (v_1^{\ast} + av_2^{\ast}) \otimes v_3 + v_2^{\ast} \otimes v_4, \qquad t_2 = (v_1^{\ast} + bv_3^{\ast}) \otimes v_2 + v_3^{\ast} \otimes v_4
 \]
for some $a,b \in \CC$ and $v_i \in H^0(\Lcal)^{\vee}$ such that the $v_i$ are linearly independent and
 \[
  [-v_2 \wedge v_3 + v_1 \wedge (av_3 + v_4)] =
  [v_2 \wedge v_3 + v_1 \wedge (bv_2 + v_4)] = \langle t_1 \rangle \cap \langle t_2 \rangle.
 \]
This equality is impossible.

\hypertarget{Case3-2-b}{\textit{Case 2-b.}} Assume that $\mathrm{dist}([Z_1], [Z_2]) = 2$.
Let $\ker(t_1) = \Jcal/\Ical_{Z_1}^2$ for some $\Ical_{Z_1}^2 \subset \Jcal \subset \Ical_{Z_1}$.
By (\hyperlink{dagger2}{$\dagger_2$}) and Lemma~\ref{lem:sections-inclusion-implies-inclusion}, $\Ical_{Z_2}^2 \subset \Jcal \subset \Ical_{Z_2}$ and $\ker(t_2) = \Jcal/\Ical_{Z_2}^2$.
In particular, we have $\Jcal = \Ical_{Z_1} \cap \Ical_{Z_2}$.
When $n = 1$, the same argument as in Case~\hyperlink{Case3-2-a}{2-a} leads to a contradiction.
When $d = 2$, we are done because $\supp(Z_1)$ and $\supp(Z_2)$ cannot be disjoint.
\end{proof}

In conclusion, one can summarize the identifiability results as follows:
\begin{enumerate}
    \item If $n = 1$ or $d = 2$, the points in $\mathrm{Sec}(X^{[d]}) \setminus X^{[d]}$ are identifiable, whence Theorem~\ref{thm:identifiable-locus-for-Hilbert-scheme-of-points}.
    \item The same proof shows that the points in
     \[
      (\mathrm{Sec}(X^{[d]}) \setminus X^{[d]}) \cap \Gr(d,H^0(\Lcal)^{\vee})
     \]
    are identifiable when $n \leq 2$ or $d \leq 3$.
\end{enumerate}


\section{Syzygies of the Hilbert scheme of two points} \label{sec:syzygies}

Under some mild conditions on $X$, we have a resolution of singularities of its secant variety as follows.

\begin{thm}[{\cite[Theorem~3.9]{Vermeire2001}}] \label{thm:Singularity-of-secant-variety}
    Let $X \subset \PP(V)$ be a nondegenerate smooth projective variety.
    Suppose that $X$ satisfies $(K_2)$, i.e., $X$ is cut out by quadrics in $\PP(V)$ and the syzygies among them are generated by linear ones.
    If $X$ contains no lines or conics, then the blowup $\mathrm{Bl}_X \mathrm{Sec}(X)$ is smooth.
    In particular, $\mathrm{Sec}(X)$ is smooth off $X$.
\end{thm}

In fact, $\mathrm{Bl}_X \mathrm{Sec}(X)$ is isomorphic to the secant bundle, see Section~\ref{sec:application}.
The imposing conditions on syzygies in Theorem~\ref{thm:Singularity-of-secant-variety} are important.
Indeed, there is an example of a complete intersection variety $X$ that does not have linear syzygies and the blowup $\mathrm{Bl}_X \mathrm{Sec}(X)$ does not separate tangent lines to $X$.

\begin{exa} \label{exa:Imposing-conditions-on-syzygies-needed}
    Let $n \geq 5$ and let $z_0, \dots, z_n$ be homogeneous coordinates on $\PP^n$.
    Consider the quadrics of the form
     \[
      Q_f = z_0^2 + z_1^2 + f,
     \]
    where $f \in H^0(\PP^n, \Ocal(2))$ is a quadric without $z_0^2$, $z_1^2$, $z_0z_2$, $z_1z_2$ and $z_2^2$ terms.
    For simplicity, denote by $V \subset H^0(\PP^n, \Ocal(2))$ the linear system of such quadrics $f$.
    Note that a family of (at least three) quadrics $Q_f$ for general $f \in V$ does not satisfy $(K_2)$ by \cite[Lemma~2.4]{Vermeire2001} because its restriction to $\PP_{z_0,z_1}^1$ is linearly dependent.
    
    Fix $3 \leq c \leq n-2$.
    Suppose that $Q_{f_1}, \dots, Q_{f_c}$ define a smooth complete intersection $X$.
    Consider $p = [0:0:1:0:\dots:0] \in X$ and lines
     \[
      \ell_0 = (z_1 = z_3 = \dots = z_n = 0), \qquad \ell_1 = (z_0 = z_3 = \dots = z_n = 0)
     \]
    tangent to $X$ at $p$.
    In the local coordinates $D_{+}(z_2) \times \PP^{c-1} \simeq \AA_{z}^n \times \PP_{w}^{c-1}$ of the blowup $\mathrm{Bl}_{X}(\PP^n)$, the points on the proper transform $\widetilde{\ell_0}$ of $\ell_0$ are
     \[
      ((t,0,\dots,0), [Q_1: \dots : Q_c]) = ((t,0,\dots,0), [1: \dots :1])
     \]
    for $t \in \CC$; and similarly, the points on the proper transform $\widetilde{\ell_1}$ of $\ell_1$ are
     \[
      ((0,t,0,\dots,0), [Q_1: \dots : Q_c]) = ((0,t,0,\dots,0), [1: \dots :1]).
     \]
    Hence the blowup does not separate those two lines.

    It remains to find such a family of quadrics.
    We will mimic the proof of Bertini's theorem \cite[Theorem~II.8.18]{Hartshorne1977}.
    For a quadric $Q$, let $H_Q$ denote the hypersurface defined by $Q$.
    Let $X$ be the variety defined by $Q_{f_1}, \dots, Q_{f_c}$ for general choices of $f_1, \dots, f_c \in V$.
    Observe that $X$ intersects the linear subspace $\PP_{z_0,z_1,z_2}^{2}$ only at a point $p = [0:0:1:0:\dots:0]$.
    Since the tangent space of $H_{f}$ at $p$ is cut out in $T_{\PP^n,p}$ by
     \[
      a_{23}dz_{3} + \dots + a_{2n}dz_n
     \]
    where $a_{2k}$ is the coefficient of $z_2z_k$ in $f$ for $3 \leq k \leq n$, the variety $X$ is smooth at $p$ provided $c \leq n-2$.
    
    Now fix a smooth subvariety $X \subset \PP^n$ of dimension at least $3$ and a point $p=[p_0: \dots : p_n] \in X \setminus \PP_{z_0,z_1,z_2}^2$.
    Consider the set
     \[
      \Sigma = \left\{ (f, q) \in V \times X \setminus \PP_{z_0,z_1,z_2}^2: q \in H_{Q_{f}} \text{ and } X \cap H_{Q_{f}} \text{ is singular at } q \right\}.
     \]
    Let $\pr_2:\Sigma \rightarrow X$ be the projection map.
    Then the set
     \[
      V_p = \left\{ f \in V: p \in H_{Q_f} \right\}
     \]
    is an affine subspace of codimension $1$ in $V$.
    Define a linear map
     \[
      \varphi_p:H^0(\PP^n, \Ocal(2)) \rightarrow \Ocal_{X,p}/\mfr_{p}^2
     \]
    as follows: fix $i$ with $p_i \neq 0$, and set $\varphi_p(Q) = \frac{Q}{z_i^2}\big|_{p} +\mfr_{p}^2$ for $Q \in H^0(\PP^n, \Ocal(2))$.
    
    As $p \not\in \PP_{z_0,z_1,z_2}^2$, we may assume that $p_3 = 1$.
    For any $0 \leq i \leq j \leq n$, one can choose a quadric
     \[
      (z_i - p_iz_3)(z_j - p_jz_3) \in \ker \varphi_p,
     \]
    which implies that $\varphi_p(z_iz_j) \in \varphi_p(V)$.
    The map $\varphi_p$ is surjective as $\Ocal(2)$ is very ample, so the restriction $\varphi_p|_V$ is also surjective.
    
    Then we have
     \[
      \Sigma_p = \pr_2^{-1}(p)
        = V_p \cap \varphi_p^{-1}(\varphi_p(-z_0^2-z_1^2))
        = V \cap \varphi_p^{-1}(\varphi_p(-z_0^2-z_1^2))
     \]
    because
     \[
      (z_0^2 + z_1^2 + V) \cap \ker\varphi_p = (z_0^2 + z_1^2 + V_p) \cap \ker\varphi_p.
     \]
    Hence $\Sigma_p$ has codimension $\dim X + 1$ in $V$, and $\Sigma$ is an irreducible variety of dimension $\dim V - 1$.
    Therefore, $\Sigma$ does not dominate $V$; i.e., the quadrics $Q_{f_1}, \dots, Q_{f_c}$ intersect transversally for general choices of $f_1, \dots, f_c \in V$.
\end{exa}

Before discussing syzygies, we first prove the linear normality of the embedding of $X^{[2]}$.
Let $\Lcal$ be a $2$-very ample line bundle on a smooth projective variety $X$ of dimension $n$.
Let $\Lcal' = \Lcal^{[2]} - \frac{B}{2}$ be the line bundle defining $\varphi_{1, \Lcal}$, where $B \subset X^{[2]}$ is the divisor parametrizing nonreduced subschemes (cf. an appendix of \cite{BS1991}).
Let $\pi:\mathrm{Bl}_{\Delta} (X \times X) \rightarrow X^{[2]}$ be the quotient map under the natural action of $\Sfr_2$.
Let $\mu:\mathrm{Bl}_{\Delta} (X \times X) \rightarrow X \times X$ be the blowup morphism and $E$ be the exceptional divisor.
Recall that $B \simeq E$ as $\pi$ is a double cover with ramification divisor $E$ and branch divisor $B$.
Denote by $f:B \simeq \PP(\Omega_X) \rightarrow X$ the projection morphism.
We have a decomposition into eigenspaces of the involution
 \[
  \pi_{\ast} \mu^{\ast} \Lcal^{\boxtimes 2} = \Lcal^{[2]} \oplus \left( \Lcal^{[2]} - \frac{B}{2} \right) = \Lcal^{[2]} \oplus \Lcal'.
 \]
Since
 \[
  H^0(X^{[2]}, \Lcal') \subset H^0(X^{[2]}, \pi_{\ast} \mu^{\ast} \Lcal^{\boxtimes 2}) \simeq H^0(X,\Lcal)^{\otimes 2}
 \]
is the $(-1)$-eigenspace, we have
 \[
  H^0(X^{[2]}, \Lcal') \simeq \bigwedge^2 H^0(X, \Lcal).
 \]
Since $\varphi_{1, \Lcal}$ is nondegenerate, it is linearly normal, i.e., given by the complete linear series $|\Lcal'|$.

Now we study defining equations and syzygies of the Hilbert scheme of two points $X^{[2]}$.
Define the adjoint line bundle
 \[
  \Lcal = K_X + dA + C
 \]
for an integer $d \gg 0$ and $A,C \in \Pic X$, where $A$ is very ample and $C$ is nef.
One can write $\Lcal'$ as
 \[
  \Lcal' = \left( K_{X^{[2]}} - \frac{n-2}{2}B \right) + \left( dA^{[2]} - \frac{B}{2} \right) + C^{[2]}
 \]
because
 \[
  \mu^{\ast} K_{X \times X} = K_{\mathrm{Bl}_{\Delta}(X \times X)} - (n-1)E = \pi^{\ast} K_{X^{[2]}} - (n-2)E.
 \]

\begin{thm} \label{thm:Syzygies-of-Hilbert-scheme-of-two-points}
    For an integer $p \geq 0$, the line bundle $\Lcal'$ satisfies $(N_p)$ provided that $d \gg 0$ and $A$ is sufficiently positive.
\end{thm}

\begin{proof}

We mimic the proof of \cite[Theorem~1]{EL1993}.
By the cohomological criterion \ref{lem:cohomological-criterion-for-np} for $(N_p)$, we need
 \[
  H^1 \left( X^{[2]}, \bigwedge^q M_{\Lcal'} \otimes \Lcal'^k \right) = 0
 \]
for all $1 \leq q \leq p+1$ and $k \geq 1$.
From the restriction sequence
 \[
  \begin{tikzcd}
      0 \arrow{r} & \Ocal_{X^{[2]}}(-B) \arrow{r} & \Ocal_{X^{[2]}} \arrow{r} & \Ocal_B \arrow{r} & 0,
  \end{tikzcd}
 \]
it suffices to prove that
 \begin{equation} \label{eqn:vanishing-on-X2}
  H^1(X^{[2]},M_{\Lcal'}^{\otimes q} \otimes \Lcal'^{k}(-r_0B)) = 0
 \end{equation}
for some fixed $r_0 \gg 0$, and
 \begin{equation} \label{eqn:vanishing-on-B}
 \begin{split}
  &H^1 \left( B, \bigwedge^q M_{\Lcal'} \otimes \Lcal'^{k}(-rB)|_B \right) \\
    &= H^1 \left( B, \bigwedge^q M_{\Lcal'}|_B \otimes f^{\ast}\Lcal^{2k} \otimes \xi^{2r+k} \right)
    = 0
 \end{split}
 \end{equation}
for $0 \leq r < r_0$, where $\xi = \Ocal_B(1)$ is the tautological line bundle.

For the vanishing (\ref{eqn:vanishing-on-X2}), we show that the line bundle $\Lcal'$ is $0$-regular with respect to $A' = A^{[2]} - \frac{B}{2}$ to apply Lemma~\ref{lem:non-exact-resolution}.
Note that
 \[
  H^{k}(X^{[2]}, \Lcal' - iA') \subset H^{k}(\mathrm{Bl}_{\Delta}(X \times X), \mu^{\ast} (\Lcal^{\boxtimes 2} - iA^{\boxtimes 2}) + (i-1)E).
 \]
We will use induction on $i$ to prove the vanishing of the right-hand side for $k \geq i \geq 1$.
The case $i = 1$ is direct due to the Kodaira vanishing theorem.
For the induction step, observe that
 \[
  \mu^{\ast} (\Lcal^{\boxtimes 2} - iA^{\boxtimes 2}) + (i-1)E \simeq K_{\mathrm{Bl}_{\Delta}(X \times X)} + \mu^{\ast}((d-i)A^{\boxtimes 2} + C^{\boxtimes 2}) + (i-n)E.
 \]
From the exact sequence
 \[
  \begin{tikzcd}
      0 \arrow{r} & \Ocal((i-1)E) \arrow{r} & \Ocal(iE) \arrow{r} & \Ocal_E(-i) \arrow{r} & 0,
  \end{tikzcd}
 \]
it suffices to show that
 \begin{align*}
  &h^{k}(E, K_{\mathrm{Bl}_{\Delta}(X \times X)}|_E + \mu^{\ast}((d-i-1)A^{\boxtimes 2} + C^{\boxtimes 2})|_E + \Ocal_E(n-i-1)) \\
  &= h^{k}(E, K_{E} + f^{\ast}(2(d-i-1)A + 2C) + \Ocal_E(n-i)) = 0
 \end{align*}
for $k \geq i+1$ by the adjunction formula.
By the Serre duality and the Leray spectral sequence, this in turn becomes
 \begin{equation} \label{eqn:positivity-of-A-required}
  \begin{split}
  &h^{2n-1-k}(E, -f^{\ast}(2(d-i-1)A + 2C) + \Ocal_E(i-n)) \\
    &= \begin{cases}
        0, & \text{if } 1 \leq i \leq n-1, \\
        h^{2n-1-k}(X, -(2(d-i-1)A + 2C) \otimes S^{i-n}\Omega_X), & \text{if } n \leq i < 2n.
    \end{cases}
  \end{split}
 \end{equation}
The vanishing follows if $A$ is sufficiently positive.

By Lemma~\ref{lem:non-exact-resolution}, there exists a (non-exact) complex
 \begin{equation} \label{eqn:non-exact-resolution-of-syzygy-bundle}
  \begin{tikzcd}
      \cdots \arrow{r} & Q_2 \arrow{r} & Q_1 \arrow{r} & Q_0 \arrow{r} & 0
  \end{tikzcd}
 \end{equation}
of vector bundles on $X^{[2]}$ such that
 \[
  Q_0 = M_{\Lcal'}^{\otimes q} \otimes \Lcal'^{k}(-r_0B)
 \]
and $Q_i$ is a direct sum of vector bundles of the form
 \[
  M_{\Lcal'}^{\otimes p} \otimes \Lcal'^{k}(-r_0B) \otimes A'^{-i}
 \]
with $p \leq q-1$.
Regarding the homology sheaves, we have: when $i \leq q-1$,
 \[
  \Hcal_i(Q_{\bullet}) = 0;
 \]
when $i \geq q$,
 \[
  \Hcal_i(Q_{\bullet}) = \bigoplus_{\substack{a_1 + \dots + a_q = i \\ a_1, \dots, a_q \geq 1}} \bigwedge^{a_1} N^{\vee} \otimes \dots \otimes \bigwedge^{a_q} N^{\vee} \otimes \Lcal'^{q+k}(-r_0B)
 \]
by the K\"unneth formula, where $N = N_{X^{[2]}/\PP(H^0(A'))}$ be the normal bundle.
Each direct summand can be written as
 \begin{align*}
   &\bigwedge^{a_1} N^{\vee} \otimes \dots \otimes \bigwedge^{a_q} N^{\vee} \otimes \Lcal'^{q+k}(-r_0B) \\
   &= \bigwedge^{e-a_1} (N \otimes A'^{\vee}) \otimes \dots \otimes \bigwedge^{e-a_q} (N \otimes A'^{\vee}) \\
   &\qquad \otimes \left[ (q+k)\Lcal'-qK_{X^{[2]}} - \left( q(2n+1) + \sum a_i \right) A' - r_0B \right]
 \end{align*}
where $e = \rank(N)$.
Here we have
 \begin{align*}
  &(q+k)\Lcal'-qK_{X^{[2]}} - \left( q(2n+1) + \sum a_i \right) A' - r_0B \\
  &= K_{X^{[2]}} + \left( (q+k)(d-2n-1) + (k+1)n+1 - \sum a_i \right) A^{[2]} \\
  &\qquad - \frac{2r_0 + n+k-2 - q(n+2) - \sum a_i}{2}B + C_0
 \end{align*}
for some nef line bundle $C_0$ because
 \[
  (K_X + (n+1)A)^{[2]} = K_{X^{[2]}} + (n+1)A^{[2]} - \frac{n-2}{2}B
 \]
is nef.
If we choose $r_0$ large enough, then one can apply the vanishing theorem of Le Potier-Sommese type (cf. \cite[Proposition~1.7]{EL1993}) to obtain
 \[
  H^{i+1}(X^{[2]}, \Hcal_i(Q_{\bullet})) = 0.
 \]
In a similar fashion, one can obtain
 \[
  H^{i}(X^{[2]}, Q_{i}) = 0
 \]
as well.
From the hypercohomology spectral sequence associated to the complex (\ref{eqn:non-exact-resolution-of-syzygy-bundle}), we get the desired vanishing (\ref{eqn:vanishing-on-X2}).
Fix such an $r_0$.

For (\ref{eqn:vanishing-on-B}), we will use the results of Park \cite{Park2007} on the syzygies of projective bundles.
Recall that a vector bundle $\Ecal$ is nef if the tautological line bundle $\Ocal_{\PP(\Ecal)}(1)$ is nef on $\PP(\Ecal)$.
He proved that the tautological line bundle on the projective bundle satisfies $(N_p)$ when the vector bundle is sufficiently positive.
More precisely:

\begin{thm} [{\cite[Theorem~1.2]{Park2007}}] \label{thm:syzygies-of-projective-bundles}
    Let $X$ be a smooth projective variety of dimension $n$, and let $\Ecal$ be a nef vector bundle of rank $r$ on $X$.
    Suppose that a very ample line bundle $A$, a nef line bundle $D$ and an integer $e \geq 0$ are given such that $A^{e} \otimes \Ecal \otimes \Ecal^{\vee}$ is nef.
    Let $\pi: \PP(\Ecal) \rightarrow X$ be the projection map and $\xi$ be the tautological line bundle.
    Then $\xi + \pi^{\ast}(K_X + fA + D)$ satisfies $(N_p)$ for $f \geq er+n+1+p$.
\end{thm}

Following the proof of Theorem~\ref{thm:syzygies-of-projective-bundles}, for fixed $q,r \geq 0$ one can obtain
 \begin{equation} \label{eqn:positivity-of-A-required-2}
  H^i \left( B, \bigwedge^q M_{\Lcal'|_B} \otimes f^{\ast}\Lcal^{2k} \otimes \xi^{2r+k} \right)
    = 0,
 \end{equation}
i.e.,
 \[
  H^i \left( B, \bigwedge^q M_{f^{\ast}\Lcal \otimes (nef) \otimes \xi} \otimes (f^{\ast}\Lcal \otimes \xi)^k \otimes (f^{\ast}\Lcal^k \otimes \xi^{2r}) \right)
    = 0
 \]
for $i, k \geq 1$ if $A$ is sufficiently positive that $\Lcal \otimes \xi^{2r}$ is nef.
Consider the following diagram
 \[
  \begin{tikzcd}
      & 0 \arrow{d} & 0 \arrow{d} & & \\
      & H^0(\Lcal'(-B)) \otimes \Ocal_B \arrow[r,equal] \arrow{d} & H^0(\Lcal'(-B)) \otimes \Ocal_B \arrow{d} & & \\
      0 \arrow{r} & M_{\Lcal'}|_B \arrow{r} \arrow{d} & H^0(\Lcal') \otimes \Ocal_B \arrow{r} \arrow{d} & \Lcal'|_B \arrow{r} \arrow[d,equal] & 0 \\
      0 \arrow{r} & M_{\Lcal'|_B} \arrow{r} \arrow{d} & H^0(\Lcal'|_B) \otimes \Ocal_B \arrow{r} \arrow{d} & \Lcal'|_B \arrow{r} & 0 \\
      & 0 & 0 & &
  \end{tikzcd}
 \]
The second column is exact because $H^1(\Lcal'(-B)) = 0$.
Let $W = H^0(\Lcal'(-B))$ for simplicity.
Then the Eagon-Northcott complex reads as
 \[
  \begin{tikzcd}[column sep=small]
      0 \arrow{r} & S^qW \otimes \Ocal_B \arrow{r} & S^{q-1}W \otimes M_{\Lcal'}|_B \arrow{r} & S^{q-2}W \otimes \bigwedge^2 M_{\Lcal'}|_B & \\
      \arrow{r} & \cdots \arrow{r} & \bigwedge^q M_{\Lcal'}|_B \arrow{r} & \bigwedge^q M_{\Lcal'|_B} \arrow{r} & 0.
  \end{tikzcd}
 \]
From the hypercohomology spectral sequence, once we show that
 \begin{equation} \label{eqn:vanishing-of-syzygy-bundle}
  H^i \left( B, \bigwedge^{q-i} M_{\Lcal'}|_B \otimes f^{\ast}\Lcal^{2k} \otimes \xi^{2r+k} \right) = 0
 \end{equation}
for $i \geq 1$ and $0 \leq r < r_0$, we would get the desired vanishing (\ref{eqn:vanishing-on-B}) as the differentials coming in and going out from
 \[
  H^1 \left( B, \bigwedge^q M_{\Lcal'}|_B \otimes f^{\ast}\Lcal^{2k} \otimes \xi^{2r+k} \right)
 \]
are all zero.
The vanishing (\ref{eqn:vanishing-of-syzygy-bundle}) is done by the induction on $i$.
Note that the assumption on the ampleness of $A$ is invoken only finitely many times.
\end{proof}

\begin{rmk} \label{rem:bound-for-positivity}
    In the proof $d \geq 3n+1$ suffices, but we do not obtain a bound for positivity of $A$ due to the vanishing of (\ref{eqn:positivity-of-A-required}) and (\ref{eqn:positivity-of-A-required-2}).
\end{rmk}

\section{Application} \label{sec:application}

We explore the geometry of $X^{[2]}$ more thoroughly, keeping the notations from the previous sections.
For a smooth projective variety $X$, consider the universal family $\Xi_2 \subset X \times X^{[2]}$ of two points.
Let $\pr_1:\Xi_2 \rightarrow X$ and $\pr_2:\Xi_2 \rightarrow X^{[2]}$ be the projection maps.
For a very ample line bundle $\Lcal$ on $X$, define the vector bundle $\Ecal_{\Lcal} = \pr_{2, \ast}(\pr_1^{\ast} \Lcal)$ of rank $2$ on $X^{[2]}$.
Let $X \subset \PP(V)$ be the embedding induced by $\Lcal$ where $V = H^0(X,\Lcal)^{\vee}$.
Consider a surjection $\Ocal_{X^{[2]}}^{\oplus (n+1)} \rightarrow \Ecal_{\Lcal}$ on $X^{[2]}$ given by $H^0(X,\Lcal) \rightarrow H^0(X,\Lcal \otimes \Ocal_{Z})$ over any $[Z] \in X^{[2]}$.
This induces an inclusion $\PP(\Ecal_{\Lcal}) \subset X^{[2]} \times \PP(V)$, and the image of the projection $u:\PP(\Ecal_{\Lcal}) \subset X^{[2]} \times \PP(V) \rightarrow \PP(V)$ is exactly the secant variety $\mathrm{Sec}(X)$ (cf. \cite[Section~3]{Vermeire2001}).

Theorem~\ref{thm:Singularity-of-secant-variety} follows by showing that $u:\PP(\Ecal_{\Lcal}) \rightarrow \mathrm{Sec}(X)$ is isomorphic to the blowup morphism $\mathrm{Bl}_{X}(\mathrm{Sec}(X)) \rightarrow \mathrm{Sec}(X)$.

For the secant variety of $X^{[2]}$, Theorem~\ref{thm:Singularity-of-secant-variety} cannot be directly applied because the blowup does not have a smooth $\PP^1$-bundle structure as it is obstructed by the existence of lines.
Nevertheless, one can adapt the proof to show that the secant variety of $X^{[2]}$ is smooth outside $X^{[2]}$ under some positivity conditions.

In the remaining of this section, we assume that $\Lcal$ is sufficiently positive (specifically, $4$-very ample) and $\Lcal'$ satisfies the property $(N_2)$.
Our first application is to show the smoothness of $\mathrm{Sec}(X^{[2]})$ outside $X^{[2]}$.

When $X$ is a curve, the positivity of the line bundle $\Lcal^{[2]}$ increases with that of $\Lcal$.
Thus \cite[Corollary~3.10]{Vermeire2001} can be directly applied.

From now on, we assume that $\dim X \geq 2$.
Let $Q_0, \dots, Q_s$ be quadrics defining $X^{[2]} \subset \PP^N = \Pbf(V)$ where $V$ is a vector space of dimension $N+1$ over $\CC$, and construct a rational map $\varphi:\PP^N \dashrightarrow \PP^s$ using the quadrics.
This yields a regular map $\widetilde{\varphi}:\widetilde{\PP^N} = \mathrm{Bl}_{X^{[2]}} \PP^N \rightarrow \PP^s$.
Let $\widetilde{\Sigma}, \widetilde{B} \subset \widetilde{\PP^N}$ be the proper transforms of $\Sigma$ and $B$, respectively.
Observe that $\varphi$ is an embedding out of $\Sigma$ by \cite[Remark~2.11]{Vermeire2001}.

\begin{lem}[cf. {\cite[Remark~3.4]{Vermeire2001}}] \label{lem:fiber-of-blowup-of-secant}
    For any $a \in \widetilde{\varphi}(\widetilde{\Sigma} \setminus \widetilde{B})$, the fiber $\widetilde{\varphi}^{-1}(a)$ maps isomorphically to a line in $\PP^N$ under the projection map $\pi_1:\PP^N \times \PP^s \rightarrow \PP^N$.
\end{lem}

\begin{proof}
    Define $T \subset \PP_z^N \times \PP_t^s$ by $\sum_{k=0}^s a_{\ell k}t_k$, where the vectors $(a_{\ell 0}, \dots, a_{\ell s})$ for $0 \leq \ell \leq r$ generate the linear syzygies among $Q_i$'s.
    Then for $a \in \widetilde{\varphi}(\widetilde{\Sigma} \setminus \widetilde{B})$, by following \cite[Proposition~2.8]{Vermeire2001}, we have either
    \begin{enumerate}
        \item $\widetilde{\varphi}^{-1}(a)$ is a reduced point;
        \item $\pi_1(\widetilde{\varphi}^{-1}(a)) = \PP^k \subset \PP^N$ for some $k > 0$ and it intersects $X^{[2]}$ in a quadric not entirely contained in $B$; or
        \item $\widetilde{\varphi}^{-1}(a) \subset T_a$ and $\pi_1(T_a)$ is a linear space in $X^{[2]}$ not entirely contained in $B$.
    \end{enumerate}
    
    Suppose that $k > 1$ in the case (2).
    By taking general hyperplane sections, we may assume that $X^{[2]}$ meets a plane $\PP^2 \subset \PP^N$ in a conic $C$ that is not entirely contained in $B$.
    If $C$ is irreducible and reduced, choose four general points $p_1, \dots, p_4$ from $C$.
    Then two secant lines $\ell = \langle p_1, p_2 \rangle$ and $\ell' = \langle p_3, p_4 \rangle$ meet.
    According to Proposition~\ref{prop:Secant-lines-intersect}, one of them lies in $X^{[2]}$, say $\ell$ does.
    But then $C$ cannot be irreducible, a contradiction.
    If $C$ is a union of two lines, by Lemma~\ref{lem:lines-in-Hilbert-scheme}, lines in $X^{[2]}$ are exactly those in the fiber of the projective bundle $B\simeq \PP(\Omega_X) \rightarrow X$.
    This leads to a contradiction.
    If $C$ is a double line, any lines in $\PP^2$ are tangent to $X^{[2]}$, and those tangent lines intersect.
    Thus $\PP^2$ itself is contained in $B$ due to Proposition~\ref{prop:Tangent-lines-intersect}, which is absurd.
    
    In conclusion, the fibers of $\widetilde{\varphi}$ project to either lines or points.
    It is noteworthy that the proper transforms of secant lines are contracted by $\widetilde{\varphi}$; indeed, two quadrics on $\PP^1$ sharing two roots are proportional.
    Thus all the fibers of $\widetilde{\varphi}$ project to lines by the semicontinuity of fiber dimension.
\end{proof}

Observe that for a secant line or a tangent line $\ell$ not in $X^{[2]}$, its proper transform does not meet $\widetilde{B}$.
Thus for $a \in \widetilde{\varphi}(\widetilde{\Sigma} \setminus \widetilde{B})$, the fiber $\widetilde{\varphi}^{-1}(a)$ is outside $\widetilde{B}$.

According to \cite[Lemma~3.5-3.7]{Vermeire2001}, one can deduce that the image of $\widetilde{\varphi}|_{\widetilde{\Sigma} \setminus \widetilde{B}}$ coincides with $(X^{[2]})_{\mathrm{id}}^{[2]} = (X^{[2]})^{[2]} \setminus \Hcal_{B/X}^{2}$, where $\Hcal_{B/X}^{2}$ is the Hilbert scheme of $0$-cycles of length $2$ lying on the fibers of $B \rightarrow X$.

Consider the universal family of lines
 \[
  \Ucal = \{ (p,[\ell]) \in \PP^N \times \Gr(2,V): p \in \ell \} \subset \PP^N \times \Gr(2,V).
 \]
Then $\Ucal$ is the projective bundle over $\Gr(2,V)$; in fact, $\Ucal \simeq \PP(U^{\vee})$ where $U$ is the universal subbundle on $\Gr(2,V)$.
By construction, the map $\widetilde{\varphi}|_{\widetilde{\Sigma} \setminus \widetilde{B}}$ is the restriction of the projection map $\Ucal \rightarrow \Gr(2,V)$ to $(X^{[2]})_{\mathrm{id}}^{[2]}$.
Observe that $\widetilde{\varphi}_{\ast} \Ocal_{\widetilde{\Sigma} \setminus \widetilde{B}}(H) \simeq U^{\vee}|_{(X^{[2]})_{\mathrm{id}}^{[2]}}$, where $H$ is the pullback of the hyperplane section of $\PP^N$.
Then the universal family $\Xi_2 \subset X^{[2]} \times (X^{[2]})_{\mathrm{id}}^{[2]}$ of two points is embedded into $\widetilde{\Sigma} \setminus \widetilde{B}$ over $(X^{[2]})_{\mathrm{id}}^{[2]}$ in such a way that the fibers over $(X^{[2]})_{\mathrm{id}}^{[2]}$ can be regarded as consisting of two points in $X^{[2]}$ lying on a (secant or tangent) line not entirely contained in $X^{[2]}$.
This gives rise to a surjection of vector bundles
 \[
  U^{\vee}|_{(X^{[2]})_{\mathrm{id}}^{[2]}} \twoheadrightarrow (\pi_2)_{\ast}(\Ocal_{\Xi_2}(H)) \simeq \Ecal_{\Lcal'}|_{(X^{[2]})_{\mathrm{id}}^{[2]}},
 \]
where $\pi_2:\PP^N \times (X^{[2]})^{[2]} \rightarrow (X^{[2]})^{[2]}$ is the projection map.
Since they have both rank $2$, this map is in fact an isomorphism.
Arguing as in \cite[Theorem~3.9]{Vermeire2001} yields an isomorphism $\widetilde{\Sigma} \setminus \widetilde{B} \simeq \PP_{(X^{[2]})_{\mathrm{id}}^{[2]}}(\Ecal_{\Lcal'})$.

In summary, one can deduce the following:

\begin{thm} \label{thm:smoothness-of-secant}
    Under the above setting, $\mathrm{Sec}(X^{[2]})$ is smooth outside $X^{[2]}$.
\end{thm}

Now assume that $X$ is a surface.
Our second application is to analyze the geometry of a resolution of singularities $\PP(\Ecal_{\Lcal'}) \rightarrow \Sigma = \mathrm{Sec}(X^{[2]})$.

Akin to Lemma~\ref{lem:fiber-of-blowup-of-secant}, the image of $\widetilde{\varphi}$ coincides with the image of the map
 \[
  \varphi_{1, \Lcal'}: (X^{[2]})^{[2]} \rightarrow \Gr(2,V),
 \]
which is defined by sending $[Z] \in (X^{[2]})^{[2]}$ to the line in $\PP^N$ determined by $Z$.
Observe that by Theorem~\ref{thm:embedding-into-Grassmannian}, the map $\varphi_{1, \Lcal'}$ might not be an embedding as $\Lcal'$ is not $2$-very ample.

Let
 \[
  \widetilde{\widetilde{\Sigma}} = \PP_{(X^{[2]})^{[2]}}(\varphi_{1, \Lcal'}^{\ast}U^{\vee})
    \subset \PP^N \times (X^{[2]})^{[2]}
 \]
be the fibered product
 \[
  \begin{tikzcd}
      \widetilde{\widetilde{\Sigma}} \arrow{r} \arrow{d} & (X^{[2]})^{[2]} \arrow[d, "\varphi_{1, \Lcal'}"] \\
      \widetilde{\Sigma} \arrow[r, "\widetilde{\varphi}"] & \im \varphi_{1, \Lcal'}.
  \end{tikzcd}
 \]
Arguing as before, we yield an isomorphism $\widetilde{\widetilde{\Sigma}} \simeq \PP(\Ecal_{\Lcal'})$.

To analyze the map $\PP(\Ecal_{\Lcal'}) \rightarrow \widetilde{\Sigma}$, we examine the geometry of the map $\varphi_{1, \Lcal'}$ further.
For $[Z] \in (X^{[2]})^{[2]}$, denote by $\langle Z \rangle \subset \PP^N$ the corresponding line.
If $\langle Z \rangle$ is not contained in $X^{[2]}$, then the set-theoretic fiber is a singleton $[Z]$.
On the other hand, if $\langle Z \rangle$ lies in $X^{[2]}$, then the set-theoretic fiber is the set of $0$-cycles of length $2$ on $\PP^1$, i.e., $(\PP^1)^{[2]} \simeq \PP^2$.

The scheme structure of the fibers can be analyzed via the following lemma:

\begin{lem}[cf. {\cite[Corollary~1.2]{CG1990}}] \label{lem:Hilbert-scheme-to-Grassmannian}
    Let $X$ be a smooth projective variety, and let $\Lcal$ be a $(d-1)$-very ample line bundle on $X$.
    Let $Z$ be a length $d$ subscheme on $X$, and let $f:\Ical_Z/\Ical_Z^2 \rightarrow \Ocal_Z$ be a nonzero $\Ocal_Z$-homomorphism.
    Then the induced linear map $d\varphi_{d, \Lcal}(f):H^0(X, \Lcal \otimes \Ical_Z) \rightarrow H^0(X, \Lcal \otimes \Ocal_Z)$ is nonzero if and only if there is a length $d+1$ subscheme $W$ on $X$ such that
    \begin{enumerate}
        \item $\Ical_{Z}^2 \subset \Ical_W \subset \Ical_Z$;
        \item $\ker(f) \subset \Ical_W/\Ical_Z^2$;
        \item the map $H^0(X, \Lcal) \rightarrow H^0(X, \Lcal \otimes \Ocal_{W})$ is surjective.
    \end{enumerate}
\end{lem}

If $\langle Z \rangle \not\subset X^{[2]}$, the line intersects $X^{[2]}$ in at most two points because $\Lcal'$ satisfies $(N_2)$.
Thus for any choices of $W$ satisfying condition (1) in Lemma~\ref{lem:Hilbert-scheme-to-Grassmannian}, the linear space spanned by $W$ must be a plane.
As a result, the fiber $\varphi_{1, \Lcal'}^{-1}(\varphi_{1, \Lcal'}([Z]))$ is a reduced point $[Z]$.

Assume that $\langle Z \rangle \subset X^{[2]}$.
If $Z$ is reduced, there are exactly two length 3 subschemes $W_1, W_2 \subset \langle Z \rangle$ that has the same support as $Z$.
Then there are tangent vectors $f_i \in T_{(X^{[2]})^{[2]}, [Z]}$ such that $\ker(f_i) = \Ical_{W_i}/\Ical_Z^2$ and $d\varphi_{1, \Lcal'}(f_i)$ is zero.
These two give rise to a $2$-dimensional family of tangent directions that are collapsed by $\varphi_{1, \Lcal'}$.
If we choose $W$ satisfying condition (1) in Lemma~\ref{lem:Hilbert-scheme-to-Grassmannian} out of these two, the linear space spanned by $W$ is a plane.
Thus the fiber $\varphi_{1, \Lcal'}^{-1}(\varphi_{1, \Lcal'}([Z]))$ is reduced at $[Z]$.

When $Z$ is nonreduced, if a tangent vector $f \in T_{(X^{[2]})^{[2]}, [Z]}$ is given such that $\ker(f) = \Jcal/\Ical_Z^2$ for some $\Ical_Z^2 \subset \Jcal \subset \Ical_Z$ and the subscheme $W$ defined by $\Jcal$ is curvilinear, then $d\varphi_{1, \Lcal'}(f) = 0$ as $W \subset \langle Z \rangle$.
These yield a $2$-dimensional family of tangent directions collapsed by $\varphi_{1, \Lcal'}$.
If $f$ is chosen outside of the family, $d\varphi_{1, \Lcal'}(f)$ is nonzero as above.
This proves that the fiber $\varphi_{1, \Lcal'}^{-1}(\varphi_{1, \Lcal'}([Z]))$ is reduced.

In conclusion, one can describe the geometry of a resolution of singularities $\PP(\Ecal_{\Lcal'}) \rightarrow \Sigma$ as follows:

\begin{prop} \label{prop:resolution-of-singularities-of-secant}
    We have a diagram
     \[
      \begin{tikzcd}
          \widetilde{\widetilde{\Sigma}} \arrow[r, "\simeq"] \arrow[rd] \arrow[d] & \PP(\Ecal_{\Lcal'}) \arrow{d} \\
          \widetilde{\Sigma} \arrow[rd, "\widetilde{\varphi}"] \arrow{d} & (X^{[2]})^{[2]} \arrow[d, "\varphi_{1, \Lcal'}"] \\
          \Sigma \arrow[r, dashrightarrow] & \im \varphi_{1, \Lcal'}
      \end{tikzcd}
     \]
    The parallelogram at the center is cartesian.
    Furthermore,
    \begin{enumerate}
        \item If $[Z] \in (X^{[2]})^{[2]}$ corresponds to a line not lying in $X^{[2]}$, then the fiber $\varphi_{1, \Lcal'}^{-1}(\varphi_{1, \Lcal'}([Z]))$ is a reduced point;
        \item Otherwise, the fiber is isomorphic to $\PP^2$.
    \end{enumerate}
\end{prop}

\section{Open questions} \label{sec:questions}

In conclusion, we raise some questions on the Hilbert scheme of points.

Our primary focus was on the Hilbert scheme of two points.
We propose the following natural question:

\begin{ques}
    Determine the identifiable locus of $\mathrm{Sec}(X^{[3]})$.
\end{ques}

We expect that points on a secant line $\langle [Z_1], [Z_2] \rangle$ (of Hamming distance $2$) are non-identifiable if and only if $Z_1$, $Z_2$ are supported in the same one point of $X$.
Similarly, we expect that points on a tangent line at $[Z]$ (of rank $2$) are non-identifiable if and only if $Z$ is supported in one point.

In the proof of Theorem~\ref{thm:Syzygies-of-Hilbert-scheme-of-two-points}, the assumption of positivity of $\Lcal$ was crucial to ensure the vanishing of (\ref{eqn:positivity-of-A-required}) and (\ref{eqn:positivity-of-A-required-2}).
We suspect that there is a bound for positivity of $\Lcal$.

\begin{ques}
    Is there a bound $d_0 = d_0(n)$ such that the line bundle $\Lcal = K_X + dA + C$ satisfies Theorem~\ref{thm:Syzygies-of-Hilbert-scheme-of-two-points} for $d \geq d_0$, given that $A$ is very ample and $C$ is nef? 
\end{ques}

When generalizing Proposition~\ref{prop:resolution-of-singularities-of-secant} to higher dimensional case, a fundamental issue arises as $X^{[2]}$ contains a higher-dimensional linear space rather than a line.
This obstructs the semicontinuity argument in Lemma~\ref{lem:fiber-of-blowup-of-secant}, leading to the following question:

\begin{ques}
    Prove that $\widetilde{\varphi}$ admits a $\PP^1$-bundle structure for higher dimensional case in Lemma~\ref{lem:fiber-of-blowup-of-secant}.
\end{ques}

Finally, one may inquire about the normality of the secant variety $\mathrm{Sec}(X^{[2]})$.
This kind of question was settled by Ullery \cite{Ullery2016}, building upon the description of $\PP(\Ecal_{\Lcal}) \rightarrow \mathrm{Sec}(X)$ in \cite{Vermeire2001}.
With a description in Proposition~\ref{prop:resolution-of-singularities-of-secant}, we hope to apply methods akin to those in \cite{Ullery2016} to prove the normality.

\begin{ques}
    Is $\mathrm{Sec}(X^{[2]})$ normal when $X$ is a surface?
\end{ques}

\section{Declaration of generative AI and AI-assisted technologies in the writing process}

During the preparation of this work the authors used ChatGPT in order to improve the readability.
After using this tool, the authors reviewed and edited the content as needed and take full responsibility for the content of the publication.

\bibliographystyle{amsplain}

\end{document}